\documentclass{amsart}

\usepackage{hyperref}
\hypersetup{hidelinks}
\usepackage{amsfonts, amsmath, amsthm, amssymb}
\usepackage{amscd, mathdots}
\usepackage[capitalize]{cleveref}
\usepackage{enumitem}

\newtheorem{thm}{Theorem}[section]
\newtheorem{lem}[thm]{Lemma}
\newtheorem{cor}[thm]{Corollary}
\newtheorem{prop}[thm]{Proposition}
\numberwithin{equation}{section}
\numberwithin{thm}{section}

\theoremstyle{definition}
\newtheorem{definition}[thm]{Definition}
\newtheorem{observation}[thm]{Observation}

\setenumerate[1]{label={(\roman*)}}

\newcommand{\C}{\mathbb{C}}
\newcommand{\R}{\mathbb{R}}

\newcommand{\RH}{\widehat{\R}}
\newcommand{\D}{\mathbb{D}}
\newcommand{\T}{\partial\D}
\newcommand{\ClD}{\overline{\D}}
\newcommand{\N}{\mathbb{N}}
\newcommand{\U}{\mathbb{H}}
\newcommand{\A}{\mathcal{S}}
\newcommand{\K}{\mathcal{K}}

\newcommand{\p}{\varphi}
\newcommand{\pt}{\widetilde{\p}}
\newcommand{\pet}{\widetilde{\p_e}}
\newcommand{\s}{\sigma}
\newcommand{\st}{\widetilde{\s}}

\newcommand{\ol}{\overline}
\newcommand{\modK}{\quad \text{(mod $\K$)}}

\newcommand{\CF}{Carath\' eodory-Fej\' er }
\newcommand{\Pick}{\mathcal{P}}
\newcommand{\Pz}{\Pick_0}
\newcommand{\dCFP}{$\partial CF \mathcal{P}$}

\newcommand{\fdbf}{\cref{Faa di Bruno}(\nameref{Faa di Bruno})}

\DeclareMathOperator{\im}{Im}
\DeclareMathOperator{\re}{\operatorname{Re} }

\begin{document}

\title[Essentially Normal Composition Operators on $H^2$]{Essentially Normal Composition Operators on $H^2$}
\author{Mor Katz}
\address{Department of Mathematics, University of Virginia, PO Box 400137, Charlottesville, VA 22904-4137, United States}
\date{March 17, 2015}

\begin{abstract}
We prove a simple criterion for essential normality of composition operators on the Hardy space induced by maps  in a reasonably large class $\A$ of analytic self-maps of the unit disk.
By combining this criterion with boundary Carath\' eodory-Fej\' er  interpolation theory, we exhibit a parametrization for all rational self-maps of the unit disk which induce essentially normal composition operators.
\end{abstract}

\email{md3ny@virginia.edu}
\maketitle

\section{Introduction}
For $\p$ an analytic self-map of the unit disk $\D$, the composition operator $C_\p \colon f \to f \circ \p$ induced by $\p$ is a bounded operator on the Hardy space $H^2$. A bounded operator $A$ is said to be essentially normal if its self-commutator $[A^*, A] = A^*A - AA^*$ is a compact operator, and trivially essentially normal if $A$ is either normal ($[A^*, A] = 0$) or compact. 
Normal composition operators on $H^2$ were characterized by Schwartz \cite{MR2618707} and compact composition operators by Shapiro \cite{MR881273} and, via a different criterion, by Sarason \cite{MR1044808} and Shapiro-Sundberg \cite{MR994787}; see also Cima-Matheson \cite{MR1452525}.
In \cite{MR1972190}, Bourdon-Levi-Narayan-Shapiro characterize the class of linear fractional self-maps of $\D$ that induce non-trivially essentially normal composition operators; these maps are exactly the parabolic non-automorphisms of $\D$. These authors provide additional examples of essentially normal composition operators induced by maps which, like linear fractional non-automorphisms of $\D$, have order of contact $2$ with $\T$ at one point.
To the best of the author's knowledge, no non-trivially essentially normal composition operators with inducing maps having order of contact $n > 2$ with $\T$ were known prior to the present work.

In this paper we prove a simple criterion (\cref{ess normality criterion} below) for essential normality of composition operators induced by maps $\p$ in the class $\A$ introduced by Kriete-Moorhouse in \cite{LRC}. Roughly speaking, the class $\A$ consists of analytic self-maps $\p$ of $\D$ that have ``significant contact'' with $\T$ at only a finite number of points, with $\p$ having ``sufficient derivative data'' at every such point. As a corollary, we show that for a self-map $\p$ of $\D$ which extends analytically to a neighborhood of $\ClD$, $C_\p$ is non-trivially essentially normal if and only if $\p$ fixes one point of $\T$, has derivative equal to $1$ there, and maps the rest of $\T$ into $\D$. Note that for the linear fractional case this criterion is equivalent to the characterization described above. Our criterion, in conjunction with boundary \CF interpolation theory, yields a parametrization for all rational self-maps $\p$ of $\D$ that induce non-trivially essentially normal $C_\p$ on $H^2$.

We rely on results from three distinct areas, presented in Sections 3-5. First, we explore a special case of a boundary version of the \CF problem studied by Agler-Lykova-Young in \cite{CFP, MR2834887}. Second, we discuss relations in the Calkin Algebra using results by Kriete-Moorhouse \cite{LRC}. In particular, we derive a decomposition of a composition operator modulo the ideal $\K$ of compact operators into a sum of composition operators induced by ``basic'' rational functions (\cref{sum decomposition into basics}). Third, using formulas and ideas from Bourdon-Shapiro \cite{BS}, based on work of Cowen-Gallardo \cite{MR2253727} and Hammond-Moorhouse-Robbins \cite{MR2394110}, we obtain an operator formula for $C_{\psi}C_\p^*$ where $\p$ is rational and $\psi$ is an auxiliary map, and reduce this formula modulo $\K$. Additionally, Fa\`a di Bruno's formula, an identity generalizing the chain rule, plays a significant role.

The author thanks her advisor, Thomas Kriete, for sharing his vision and for his continuous guidance, and Paul Bourdon for his insightful suggestions. She also thanks Vladimir Bolotnikov for sharing his ideas about boundary interpolation on the unit disk.

\section{Preliminaries}

\subsection{Generalized Chain Rule - Fa\`a di Bruno's Formula} \label{Fa di Bruno section}
$ $

Fa\`a di Bruno's formula is an identity generalizing the chain rule that has been known since 1800. The following is the statement of the formula in combinatorial form.

\begin{thm}[Fa\`a di Bruno's formula]\cite{MR1903577}\label{Faa di Bruno}
If $g$ is analytic at $z$ and $f$ is analytic at $g(z)$, then
\begin{align} \label{Faa di Bruno less useful formula}
(f\circ g)^{(k)}(z)=\sum_{\pi\in\Pi} f^{(\left|\pi\right|)}(g(z))\cdot\prod_{B\in\pi}g^{(\left|B\right|)}(z),
\end{align}
where $\Pi$ is the set of partitions of \{$1$, ..., $k$\}.
\end{thm}

We define the $n^{th}$ order data of a function $h$ at a point $z$ to be the vector
\[ D_n (h, z) = (h(z), h'(z), ..., h^{(n)}(z)).\]
Note that as a consequence of \nameref{Faa di Bruno}, $D_k(f \circ g, z)$ is determined by $D_k(g, z)$ and $D_k(f, g(z))$ when $g$ is analytic at $z$ and $f$ is analytic at $g(z)$. Furthermore, in this case we can rewrite \cref{Faa di Bruno less useful formula} as
\begin{align}  \label{useful Faa di Bruno}
(f \circ g)^{(k)}(z) = f^{(k)}(g(z))g'(z)^k + F(D_{k-1}(f, g(z)), D_{k-1}(g,z)) + f'(g(z))g^{(k)}(z),
\end{align}
where the first term originates from the partition $\pi = \{ \{1\}, \{2\},...,\{k\} \}$, the last term originates from the partition $\pi = \{ \{1, 2,...,k \} \}$, and $F$ is defined by
\[ F((a_0, a_1, ..., a_{k-1}), (b_0, b_1, ..., b_{k-1}))
= \sum_{\substack{ \pi \in \Pi \\ 1 < |\pi| < k }} a_{|\pi|}\cdot\prod_{B\in\pi}b_{|B|}. \]

\subsection{Order of Contact} \label{order of contact subsection}
$ $

Our treatment of order of contact is similar to that in \cite{LRC}. We define the notion of order of contact with the boundary both in the context of the unit disk $\D$ and in the context of the upper half-plane $\U$.

\begin{definition} Let $V$ be a neighborhood of some $\zeta \in \T$ and let $\p$ be analytic on $V \cap \D$ satisfying $\p(V \cap \D) \subset \D$. We say that $\p$ has contact with $\T$ at $\zeta$ of order $c > 0$ if the following conditions hold.
\begin{enumerate}
\item $\displaystyle \p(\zeta) := \lim_{z \to \zeta, z \in V \cap \D} \p(z)$ exists and $\p(\zeta) \in \T$.
\item $\dfrac{1-|\p(e^{i\theta})|^2}{|\p(\zeta) - \p(e^{i\theta})|^c}$

is essentially bounded above and away from zero as $e^{i\theta}\to\zeta$.
\end{enumerate}

A conformal mapping argument together with an application of Fatou's theorem shows that the non-tangential boundary values $\p(e^{i\theta})$ exist a.e.\ on $V \cap \T$, and similarly for the half-plane definition below.
\end{definition}

\begin{definition} Let $W$ be a neighborhood of $0$ and let $f$ be analytic on $W \cap \U$ satisfying $f(W \cap \U) \subset \U$. We say that $f$ has contact with $\R$ at $0$ if the following conditions hold.
\begin{enumerate}
\item $\displaystyle f(0) := \lim_{z \to 0, z \in W \cap \U} f(z)$ exists and $f(0) \in \R$.
\item $\dfrac{\im f(x)}{|f(0) - f(x)|^c}$

is essentially bounded above and away from zero as $x \to 0$ in $\R$.
\end{enumerate}
\end{definition}

We transfer $\D$ to $\U$ using the family of conformal maps $\tau_\alpha \colon \D \to \U$ for $\alpha \in \T$, defined by
$ \tau_\alpha \colon z \mapsto i\frac{\alpha-z}{\alpha+z}$. Note that
$\p$ has order of contact $c$ with $\T$ at $\zeta$ if and only if
$ f = \tau_{\p(\zeta)} \circ \p \circ \tau_\zeta^{-1}$ has order of contact $c$ with $\R$ at $0$.

Order of contact is more easily understood in the context of the upper half-plane. To gain some intuition, suppose that $f$ maps $W \cap \U$ into $\U$ for some neighborhood $W$ of $0$ and is analytic at $0$ with Taylor expansion $f(z) = \sum_{k=0}^{\infty}a_k z^k$ there. By taking imaginary parts, we see that $f$ maps an interval in $\R$ containing $0$ into $\R$ if and only if all the coefficients $a_k$ are real. Otherwise, for $n = \min \{ k \colon \im a_k \neq 0 \}$ we have that $0 < \im f(x) \sim \im(a_n) x^n$ as $x \to 0$ in $\R$. We see that $n$ must be even, $\im a_n >0$, and 
\[ \dfrac{\im f(x)}{|f(0) - f(x)|^n} \sim \dfrac{\im(a_n) x^n}{|a_1|^n x^n} = const,\]
so that $f$ has order of contact $n$ with $\R$ at $0$ (see discussion in \cite{LRC}).

We use Fa\`a di Bruno's Formula to determine the order of contact of composite maps which extend analytically to a neighborhood of the point of contact.

\begin{prop}\label{contact for composite} Let $f_1$ and $f_2$ be functions mapping $W \cap \U$ into $\U$ for some neighborhood $W$ of $0$ that are analytic at $0$, and suppose that $f_1$ fixes $0$. Then the following statements hold.
\begin{enumerate}
\item If for $i = 1, 2$, $f_i$ has order of contact $N_i$ with $\R$ at $0$, then $f_2 \circ f_1$ has order of contact equal to $\min (N_1, N_2)$ with $\R$ at $0$.
\item If one of $f_1$ and $f_2$ has order of contact $N$ with $\R$ at $0$ and the other maps an interval of $\R$ containing $0$ into $\R$, then $f_2 \circ f_1$ has order of contact $N$ with $\R$ at $0$.
\end{enumerate}
\end{prop}
\begin{proof}
Let $f_1(z) = \sum_{j=0}^{j=\infty} a_j z^j$ and $f_2(z) = \sum_{j=0}^{j=\infty} b_j z^j$ be the Taylor series of $f_1$ and $f_2$ about $0$. Let $N_1 = \min \{ k \colon \im a_k \neq 0 \}$ and $N_2 = \min \{ k \colon \im b_k \neq 0 \}$, allowing for $\infty$ in the case where all coefficients are real. Let $K = \min \{ N_1, N_2 \}$, so that $K$ is finite under the assumption of either of the statements we aim to prove.

Since $a_1, ..., a_{K-1}, b_1, ..., b_{K-1} \in \R$, using \fdbf ~and induction, we get that ${(f_2 \circ f_1)'(0)}, ..., {(f_2 \circ f_1)^{(K-1)}(0)}$ are all real valued. By \cref{useful Faa di Bruno} at $z=0$ we have
\[ (f_2 \circ f_1)^{(K)}(0)
= f_2^{(K)}(0) f_1'(0)^K + F(D_{K-1}(f_2, 0), D_{K-1}(f_1,0)) + f_2'(0)f_1^{(K)}(0),\]
and taking imaginary parts we get
\[ \im (f_2 \circ f_1)^{(K)}(0) = f_1'(0)^K \im f_2^{(K)}(0) + f_2'(0) \im f_1^{(K)}(0). \]
Since $a_1, b_1 > 0$, and $\im a_K, \im b_K \geq 0$ with at least one of them positive by the definition of $K$, it follows that $\im (f_2 \circ f_1)^{(K)}(0) >0$. Thus $f_2 \circ f_1$ has order of contact $K$ with $\R$ at $0$.
\end{proof}

\subsection{The Class of Functions $\A$} \label{the class A}
$ $

We work in the class of functions $\A$, introduced by Kriete-Moorhouse in \cite{LRC}, consisting of analytic self-maps $\p$ of $\D$ with certain properties of boundary regularity. The motivating model for a function in $\A$ is an analytic self-map of $\D$ which extends analytically to a neighborhood of $\ClD$ and is not a finite Blaschke product. In particular, we restrict the number of points of significant contact with the unit circle, and require relatively nice behavior at these points of significant contact. To make a precise definition, we first discuss Clark measures. For $\zeta \in \T$, we let $\p(\zeta)$ denote the non-tangential limit of $\p$ at $\zeta$, and $\p'(\zeta)$ denote the angular derivative of $\p$ at $\zeta$. If $\p'(\zeta)$ does not exist we say that $|\p'(\zeta)| = \infty$. If $\p'(\zeta)$ does exists, then $\p'(\zeta) = \ol{\zeta}\p(\zeta)|\p'(\zeta)|$, and in particular $\p'(\zeta) > 0$ if $\p$ fixes $\zeta$ \cite{MR1397026}.
 
Let $\p$ be an analytic self-map of $\D$. If $|\alpha| = 1$, there exists a finite positive Borel measure $\mu_\alpha$ on $\T$ such that 
\[ \frac{1-|\p(z)|^2}{|\alpha - \p(z)|^2} = \re \left(\frac{\alpha+\p(z)}{\alpha - \p(z)} \right)
= \int_{\T}P_z(e^{it})d\mu_\alpha(t)\]
for $z$ in $\D$, where $P_z(e^{it}) = \frac{1-|z|^2}{|e^{it}-z|^2}$
is the Poisson kernel at $z$. The existence of $\mu_\alpha$ follows since the left side the equation above is a positive harmonic function. The measures $\mu_\alpha$ are called the Clark measures of $\p$ (see \cite{MR2215991}, \cite{MR1289670}).

The singular part of the measure, $\mu^s_\alpha$, is carried by $\p^{-1}(\{\alpha\})$, the set of those $\zeta \in \T$ where $\p(\zeta)$ exists and equals $\alpha$. The measure $\mu^s_\alpha$ is the sum of the pure point measure $\mu^{pp}_\alpha = \sum_{\p(\zeta)=\alpha} \frac{1}{|\p'(\zeta)|}\delta_\zeta$,
where $\delta_\zeta$ is the unit point mass at $\zeta$, and a continuous singular measure $\mu^{cs}_\alpha$, either of which can be zero.
We write $E(\p) = \ol{ \bigcup_{|\alpha|=1} spt(\mu^s_\alpha) }$,
where $spt(\mu)$ denotes the closed support of a measure $\mu$, and note that for any $\p$,
$F(\p) = \{ \zeta \colon \text{$\p$ has finite angular derivative at $\zeta$} \}$
is a subset of $E(\p)$. Furthermore, if $E(\p)$ is finite then the continuous singular measures $\mu^{cs}_\alpha$ all vanish, and we get that $E(\p) = F(\p)$. For any $\p$, $C_\p$ is a compact operator on $H^2$ if and only if $E(\p)$ is the empty set, see Sarason \cite{MR1044808}, Shapiro-Sundberg \cite{MR994787}, and Cima-Matheson \cite{MR2215991}.

\begin{definition} We define the class $\A$ to be the set of analytic self-maps $\p$ of $\D$ satisfying the following conditions.
\begin{enumerate}
\item \label{DefA1} $| \p(e^{i \theta})| < 1$ a.e.\ on $\T$.
\item \label{DefA2} $E(\p)$ is a finite set, so that $E(\p) = F(\p)$.
\item \label{DefA3} For each point $\zeta \in F(\p)$, there exists an even positive integer $n$ such that $\p$ has order of contact $n$ at $\zeta$, and complex numbers $a_0,a_1,...,a_n$ with
\[ \p(z) = a_0 +a_1(z-\zeta)+...+ a_{n}(z-\zeta)^{n} +o(|z-\zeta|^{n}) \]
as $z \to \zeta$ unrestrictedly in $\D$.
\end{enumerate}
Note: it can be shown that for any $\p$, \ref{DefA2} implies \ref{DefA1}.
\end{definition}

For $\p \in \A$ with order of contact $n$ at $\zeta$, we define the derivatives of $\p$ at $\zeta$ by
\[ \p^{(j)}(\zeta) := \angle \lim_{z \to \zeta} \p^{(j)}(z) = j!a_j \]
for $j = 1, ..., n$ and note that these non-tangential limits do exist (see the argument in \cite[p~47]{MR1289670}).

\begin{prop}\label{A is big} $\A$ contains all self-maps of $\D$ that extend analytically to a neighborhood of $\ClD$ and are not finite Blaschke products.
\end{prop}
\begin{proof}
Let $\p$ be such a map defined on a neighborhood $V$ of $\ClD$, and define $A = \{ \zeta \in \T \colon  \ |\p(\zeta)|=1 \}$ to be the set of points where $\p$ has contact with $\T$. Then $A$ consists of zeros of the analytic function $f(z) = \p(z) - \p_e(z)$, where $\p_e = \rho \circ \p \circ \rho$, and $\rho \colon z \mapsto 1/\ol{z}$ is inversion in the unit circle. To obtain a contradiction, suppose that $A$ is infinite. Then $f$ must be identically $0$ on $V$ and $\p$ must be a finite Blaschke product, contradicting our assumption. Thus $A$ is finite and $E(\p) \subset A$ is finite as well. Now let $\zeta \in F(\p) = E(\p)$. Note that $\p$ is analytic at $\zeta$ and maps a small arc containing $\zeta$ onto a curve with contact with $\T$ at exactly one point. Thus $\p$ has finite (necessarily even) order of contact, say $n$, with $\T$ at $\zeta$. To complete the proof, we use the Taylor coefficients of $\p$ at $\zeta$ to write
\[ \p(z) = a_0 +a_1(z-\zeta)+...+ a_{n}(z-\zeta)^{n} +o(|z-\zeta|^{n}). \]
\end{proof}

\subsection{The Denjoy-Wolff Point} \label{Denjoy-Wolff Point}
$ $

If $\p$ is an analytic self-map of $\D$, not the identity and not an elliptic automorphism, then $\p$ has a unique attractive fixed point $\omega$ in $\ClD$.  If $\omega$ lies on $\T$, it is characterized by $\p(\omega) = \omega$ and $0 < \p'(\omega) \leq 1$.  As above, $\p(\omega)$ is interpreted in the sense of nontangential limit and $\p'(\omega)$ is the angular derivative at $\omega$; see \cite[Section~2.3]{MR1397026}.

\section{The Boundary \CF Problem} \label{The CF Problem} \label{The CF Problem - section}

The \CF problem \cite{OriginalCFP, MoreCFP} is to determine whether a given finite sequence of complex numbers comprises the initial Taylor coefficients of an analytic map $f$ mapping the unit disk $\D$ to the upper half-plane $\U$. 
In this section, we explore a special case of a boundary version of the \CF problem studied by Agler-Lykova-Young in \cite{CFP, MR2834887}, where the functions considered are analytic self-maps of $\U$. We note that Bolotnikov studies an alternative version of boundary interpolation, where the functions considered are self-maps of $\D$ \cite{Bolotnikov2011568, Bolotnikov20123123}.

For any $x$ in $\R$ we let $\Pick_x$ denote the set of maps in $\Pick$ that extend analytically to a neighborhood of $x$, where $\Pick$ is the Pick class consisting of maps $f$ analytic on $\U$ which satisfy $\im f(z) \geq 0$ on $\U$.
In \cite{CFP}, Agler-Lykova-Young study a boundary interpolation problem, denoted \dCFP, where the interpolation node $x$ lies on $\R$ and solutions lie in $\Pick_x$. In the subsequent paper, \cite{MR2834887}, weaker solutions to \dCFP, having non-tangential pseudo-Taylor expansions, are considered. For our purposes, unrestricted pseudo-Taylor expansions will suffice. We restrict attention to the case where solutions to \dCFP ~have even order of contact $n$ with $\R$ at $0$. The following is a consequence of \cite[Theorem~1.2(2)]{CFP} and \cite[Theorem~5.2]{MR2834887}.

\begin{thm}\label{CFP condition}\cite{CFP, MR2834887} Let $n = 2m$ be an even positive integer, $a_0, ... ,a_{n-1} \in \R$ and $a_n \in \U$, and let $H_m(a_1,...,a_{n-1})$ be the Hankel matrix defined by
\[  H_m(a_1,...,a_{n-1}) = \begin{bmatrix} 
a_1 & a_2 & ... & a_m     \\
a_2 & a_3 & ... & a_{m+1} \\
.   & .   & ... & .       \\
a_m & a_{m+1} & ... & a_{n-1} \end{bmatrix}. \]
Then the following are equivalent:
\begin{enumerate}
\item There exists a function $f \in \Pz$ that has initial Taylor coefficients $a_0,...,a_n$ at $0$.
\item There exists a function $f \in \Pick$ satisfying 
\[ f (z) = a_0 + a_1 z + ... +a_n z^n +o(|z|^n) \]
as $z \to 0$ unrestrictedly in $\U$.
\item $H_m(a_1,...,a_{n-1}) >0$, i.e., this matrix is positive definite.
\end{enumerate}
\end{thm}

\subsection{Parametrization of Solutions to the Contact-$n$ Case}
$ $

In \cite{CFP}, Agler-Lykova-Young give a parametrization of all solutions in the case where $a_0, ..., a_n$ are real.
We apply the same techniques to the order of contact $n$ case, that is, the case where $a_0, ... ,a_{n-1} \in \R$ and $a_n \in \U$, and arrive at a similar parametrization. The main tool used is a technique for passing from a function in the Pick class to a simpler one and back again due to G. Julia \cite{MR1555173}. Reduction and augmentation (at $0$) of a function are defined as follows.

\begin{definition} For any non-constant function $f \in \Pz$ such that $f(0) \in \R$, we define the \textbf{reduction} of $f$ (at $0$) to be the function $g$ on $\U$ given by the equation
\[ g(z) = -\frac{1}{f(z) - f(0)} + \frac{1}{f'(0)z}. \]
\end{definition}

\begin{definition} For any function $g \in \Pz$ and any $a_0 \in \R$, $a_1 > 0$, we define the \textbf{augmentation} of $g$ (at $0$) by $a_0, a_1$ to be the function $f$ on $\U$ given by
\[ f(z) = a_0 +\frac{1}{\frac{1}{a_1 z} - g(z)}. \]
\end{definition}

Reduction and augmentation preserve the Pick class (see \cite[Theorem~3.4]{RedAug}) and are inverse operations for functions in $\Pz$, that is,
\begin{enumerate}
\item if $f \in \Pz$ is non-constant and $f(0) \in \R$ then the reduction $g$ of $f$ is in $\Pz$ as well, and $f$ is the augmentation of $g$ by $f(0), f'(0)$;
\item if $g \in \Pz$ and $a_0 \in \R$, $a_1 > 0$ then the augmentation $f$ of $g$ by $a_0, a_1$ is in $\Pz$ as well and satisfies $f(0)=a_0$ and $f'(0)=a_1$, and $g$ is the reduction of $f$.
\end{enumerate}

The relationship between the Taylor coefficients of a function and those of its reduction is explicitly expressed in \cite[Proposition~2.5]{CFP}. The following corollary to \cite[Proposition~2.5]{CFP} includes a statement contained in \cite[Corollary~3.3]{CFP}.

\begin{cor}\label{Taylor coeff relation} Let $f \in \Pz$ satisfy $f'(0) >0$, and let $g$ be the reduction of $f$. Let the Taylor expansions of $f$ and $g$ about $0$ be $f(z) = \sum_{j=0}^{\infty}a_jz^j$ and $g(z) = \sum_{j=0}^{\infty}b_jz^j$. Then the following statements hold.
\begin{enumerate}
\item \label{TCR1} For any $n \geq 2$, $a_1,...,a_n$ determine $b_0,...,b_{n-2}$ and in the other direction $a_1, b_0, ... , b_{n-2}$ determine $a_2,...a_n$.

\item \label{TCR2} If $a_0 \in \R$ then for any $k \geq 2$, $a_k$ is the first non-real Taylor coefficient of $f$ if and only if $b_{k-2}$ is the first non-real Taylor coefficient of $g$, that is to say, reduction reduces the order of contact by $2$.

\item \label{TCR3} $H_m(a_1,...,a_{2m-1}) > 0$ if and only if ${H_{m-1}(b_1,...,b_{2m-3}) > 0}$.
\end{enumerate}
\end{cor}
\begin{proof} 
\ref{TCR1} and \ref{TCR2} follow from \cite[Proposition~2.5]{CFP} and a calculation, and \ref{TCR3} is contained in \cite[Corollary~3.3]{CFP}.
\end{proof}

For a matrix $A = \begin{bmatrix} a_{11} & a_{12}\\ a_{21} & a_{22} \\ \end{bmatrix}$, we denote the corresponding linear fractional transformation by $L[A]$:
\[ L[A]h = \dfrac{a_{11}h+a_{12}}{a_{21}h+a_{22}}. \]
With this notation, introduced in \cite{CFP}, we can express the augmentation $f$ of $g$ by $a_0, a_1$ by $f(z) = L[A(a_0, a_1)(z)]g(z)$, where $A(a_0, a_1)(z)$ is defined by
\[ A(a_0, a_1)(z) = \begin{bmatrix} a_0 a_1 z &  - a_0 - a_1 z \\  a_1 z & -1 \\ \end{bmatrix}. \]

Note that composition of linear fractional transformations corresponds to matrix multiplication, and so this notation enables conversion of multiple augmentations into matrix multiplication.

\begin{thm}\label{dCFP contact param yucky} Let $n=2m$ be an even positive integer and let $a_0,...,a_{n-1} \in \R$ and $a_n \in \U$ be such that $H_m(a_1,...,a_{n-1})>0$. Let
\[ a_0 = a^{(0)}_0, a^{(1)}_0,..., a^{(m-1)}_0 \in \R, \quad a^{(m)}_0 \in \U, \qquad a_1 = a^{(0)}_1, a^{(1)}_1,...,a^{(m-1)}_1 > 0, \]
be the parameters determined by $a_0, ..., a_n$ via the procedure in the proof below. Then a functions $f \in \Pz$ has initial Taylor coefficients $a_0, ..., a_n$ if and only if $f$ is of the form
\[ f(z) = L[A(a_0^{(0)}, a^{(0)}_1)(z) \cdots A(a^{(m-1)}_0, a^{(m-1)}_1)(z)] g(z) \]
where $g \in \Pz$ and satisfies $g(0) = a^{(m)}_0$.
\end{thm}
\begin{proof} Note that by \cref{CFP condition}, there exists a function $F_0 \in \Pz$ with initial Taylor coefficients $a_0,...a_n$. We inductively define $F_{k+1} \in \Pz$ to be the reduction of $F_k$ for $k=1,...,m$, and $a_0^{(k)}, a_1^{(k)},...$ to be the Taylor coefficients of $F_k$ at $0$. Notice that for each $k$, $F_k$ is the augmentation of $F_{k+1}$ by $a_0^{(k)}, a_1^{(k)}$, so that $F_k(z) = L[A(a^{(k)}_0, a^{(k)}_1)]F_{k+1}(z)$, and so $F_0$ can be written as
\begin{align*}
F_0(z) = L[A(a^{(0)}_0, a^{(0)}_1)(z) \cdots A(a^{(m-1)}_0, a^{(m-1)}_1)(z)]F_m(z).
\end{align*}

By \cref{Taylor coeff relation}\ref{TCR1}, all the Taylor coefficients listed below are determined by $a_0,...,a_n$ and do not depend on our choice of $F_0$.
\[ \begin{CD}
F_0 @>{\text{red}}>> F_1   @>{\text{red}}>> . ~ . ~ .  @>{\text{red}}>> F_{m-1}   @>{\text{red}}>> F_m \\
a_0 @.               a_0^{(1)} @. \cdots        @.               a_0^{(m-1)} @.               \mathbf{a_0^{(m)}} \\
a_1 @.               a_1^{(1)} @. \cdots        @.               a_1^{(m-1)} \\
a_2 @.               a_2^{(1)} @. \cdots        @.               \mathbf{a_2^{(m-1)}} \\
\vdots @.            \vdots @. \iddots \\
a_{n-2} @.           \mathbf{a_{n-2}^{(1)}} \\
a_{n-1}\\
\mathbf{a_n}
\end{CD}\]
In general, for $k=1,...,m$ we have that $a^{(k)}_0,...a^{(k)}_{n-2k}$ are determined. It follows from \cref{Taylor coeff relation}\ref{TCR2} that $a^{(k)}_0,...a^{(k)}_{n-2k -1} \in \R$ and $a^{(k)}_{n-2k} \in \U$, i.e., all the non-bold coefficients above are in $\R$ and all the bold coefficient are in $\U$. In particular, we get that
\[ a_0 = a^{(0)}_0, a^{(1)}_0,..., a^{(m-1)}_0 \in \R, \qquad a^{(m)}_0 \in \U. \]
Additionally, \cref{Taylor coeff relation}\ref{TCR3} implies that for $k=1,..., m-1$ the Hankel matrix $H_{m-k}(a^{(k)}_1,...,a^{(k)}_{n-2k-1})$ is positive, and so in particular we have
\[ a_1 = a^{(0)}_1, a^{(1)}_1,...,a^{(m-1)}_1 > 0. \]
Note that since $F_0, ..., F_{m-1}$ are all real valued at $0$ and non-constant, taking the above reductions makes sense.

To prove the forward implication, suppose $f = f_0 \in \Pz$ has initial Taylor coefficients $a_0,...,a_n$, and let $f_k$ denote the $k^{th}$ reduction of $f_0$. As discussed above, the first $n-2k +1$ Taylor coefficients of $f_k$ are $a^{(k)}_0,...a^{(k)}_{n-2k}$, and so $f_0$ can be written as
\[ f_0(z) = L[A(a^{(0)}_0, a^{(0)}_1)(z) \cdots A(a^{(m-1)}_0, a^{(m-1)}_1)(z)]f_m(z), \]
with $f_m$ satisfying $f_m(0) = a_0^{(m)}$.

To prove the backward implication, let $g = f_m \in \Pz$ satisfy $g(0)= a^{(m)}_0$ and inductively define $f_k$ for $k = m-1 ,...,0$ to be the augmentation of $f_{k+1}$ by $a_0^{(k)}, a_1^{(k)}$, i.e.,
$f_k(z) = L[A(a^{(k)}_0, a^{(k)}_1)]f_{k+1}(z)$. In order to complete the proof note that it follows from \cref{Taylor coeff relation}\ref{TCR1} that for each $k$ the first $n-2k+1$ Taylor coefficients of $f_k$ are again $a^{(k)}_0,...a^{(k)}_{n-2k}$.
\end{proof}

Note that by multiplying the matrices in \cref{dCFP contact param yucky}, we get that functions $f \in \Pz$ with initial Taylor coefficients $a_0,...,a_n$ are of the form
\[ f(z) = \frac{p(z)h(z)+q(z)}{r(z)h(z)+s(z)}, \]
where $p,q,r,s$ are polynomials with real coefficients of degree at most $m$ determined by $a_0,...,a_n$
and $h \in \Pz$ satisfies $h(0) = h_0$, where $h_0$ is determined by $a_0,...,a_n$.
Additionally, a calculation of determinants shows that for some $K>0$ the polynomials $p,q,r,s$ satisfy $(ps - qr)(z) = Kz^n$.

We let the Taylor coefficients $a_0,...a_n$ vary to obtain the following corollary.

\begin{prop}\label{param of contact n} Let $f \in \Pz$. Then $f$ has order of contact $n=2m$ with $\R$ at $0$ if and only if $f$ is of the form
\[ f(z) = L[A(a_0^{(0)}, a^{(0)}_1)(z) \cdots A(a^{(m-1)}_0, a^{(m-1)}_1)(z)] g(z), \]
where $g \in \Pz$ satisfies $g(0) \in \U$ and
\[ a^{(0)}_0, a^{(1)}_0,..., a^{(m-1)}_0 \in \R, \qquad a^{(0)}_1, a^{(1)}_1,...,a^{(m-1)}_1 > 0.\]
Furthermore, for any $f \in \Pz$ this representation is unique.
\end{prop}
\begin{proof}
For the first direction, assume $f$ has order of contact $n=2m$ with $\R$ at $0$, and let $a_0,...,a_{n-1} \in \R$ and $a_n \in \U$ be the initial Taylor coefficients of $f$ at $0$. Then by \cref{CFP condition} we get that $H_m(a_1,...,a_{n-1})>0$, and so by \cref{dCFP contact param yucky} $f$ is of the desired form.

For the other direction note that any $f$ of this form is obtained by applying $m$ augmentations to $g \in \Pz$, and so $f \in \Pz$. Recall that by \cref{Taylor coeff relation}\ref{TCR2}, augmentation increases the order of contact with $\R$ by $2$, and so $g(0) \in \U$ implies that $f$ has order of contact $2m = n$.

To see uniqueness, notice that if
\[f(z) = L[A(a_0^{(0)}, a^{(0)}_1)(z) \cdots A(a^{(m-1)}_0, a^{(m-1)}_1)(z)] g(z) \]
then $a_0^{(k)} = f_k(0), a_1^{(k)} = f_k'(0)$ and $g = f_m$ where $f_k$ denotes the $k^{th}$ reduction of $f$.
\end{proof}

\subsection{Rational Functions with Specified Taylor Coefficients} \label{basic functions}
$ $

We turn our attention to construction of simple solutions to \dCFP. We construct rational maps $f \in \Pz$ which map $0$ into $\R$ and the rest of $\RH = \R \cup \{ \infty \}$ into $\U$ and have specified initial Taylor coefficients at $0$. Additionally we require that $0$ be a regular value for $f$, i.e., that $f^{-1}(\{0\})$ consists of $d$ distinct points where $d$ is the degree of $f$.

\begin{observation} Reduction and augmentation preserve rationality and boundary behavior of functions in the sense that if $g$ is the reduction of $f$, then the following assertions hold.
\begin{enumerate}
\item $g$ is a rational function of degree $d$ if and only if $f$ a rational function of degree $d+1$.
\item For any $x \in \R \setminus \{ 0 \}$, $\im g(x) > 0$ if and only if $\im f(x) > 0$.
\end{enumerate}
\end{observation}

\begin{prop}\label{existence of basic functions}Let $f$ be in $\Pick$ and suppose that for some $n = 2m$, $f$ has pseudo Taylor coefficients
\[ \lim_{ z \to 0} \frac{f^{(k)}(z)}{k!} = a_k \quad \text{for $k=0,1,...,n$}, \]
where the limits are taken unrestrictedly in $\U$, such that $a_0, a_1,...,a_{n-1} \in \R$ and $a_n \in \U$.
Then there exists a degree $m+1$ rational function $f_0 \in \Pz$ that has $0$ as a regular value, maps $\RH \setminus \{0\}$ into $\U$ and has initial Taylor coefficients $a_0,...,a_n$ at $z=0$.
\end{prop}
\begin{proof} By \cref{CFP condition}, $H_m(a_1,...,a_{n-1})>0$ and so by \cref{dCFP contact param yucky} we have that for any $g \in \Pz$ with $g(0) = a_0^{(m)}$,
\[ F(z) = L[A(a^{(0)}_0, a^{(0)}_1)(z) \cdots A(a^{(m-1)}_0, a^{(m-1)}_1)(z)]g(z), \]
is in $\Pz$ with the desired Taylor coefficients. Here $a^{(0)}_0=a_0, a^{(1)}_0,..., a^{(m-1)}_0 \in \R$, $a^{(m)}_0 \in \U$ and $a^{(0)}_1 = a_1, a^{(1)}_1,...,a^{(m-1)}_1 > 0$ are determined by $a_0,...,a_n$.

For any $w \in \U$, we define $g_w(z) = L[A(a^{(m)}_0, 1)(z)]w = a_0^{(m)} + \frac{z}{1-wz}$, so that $g$ is a degree $1$ rational function in $\Pz$. Note that $g_w$ can be written as the sum of $\im a_0^{(m)}$ and an augmentation of the constant function $w$, and so by the observation above $g_w$ maps $\RH$ into $\U$. Let $F_w$ denote the function $F$ above resulting from the choice $g = g_w$ and note that $F_w$ is obtained by applying $m$ augmentations to $g_w$. Thus $F_w$ is a degree $m+1$ rational function mapping $\RH \setminus \{ 0\}$ into $\U$.
                                                                                                       
It remains to find a $w \in \U$ such that $0$ is a regular value for $F_w$. We write
\[ \begin{bmatrix} p(z) & q(z) \\ r(z) & s(z) \end{bmatrix}
= A(a^{(0)}_0, a^{(0)}_1)(z) \cdots A(a^{(m-1)}_0, a^{(m-1)}_1)(z) A(a^{(m)}_0, 1)(z),\]
so that $p, q, r, s$ are polynomials of degree at most $m+1$, and $F_w(z) = \frac{p(z)w+q(z)}{r(z)w+s(z)}$.
Note that a calculation of determinants shows that $p, q, r, s$ satisfy $(ps-qr)(z) = (a^{(0)}_1)^2z^2 \cdots (a^{(m)}_1)^2z^2 = Kz^{n+2}$ for some $K \neq 0$, and recall that $F_w$ has degree $m+1$. An elementary argument considering degrees and common factors of polynomials shows that $F_w(z) = 0$ has $m+1$ distinct solutions for all but finitely many choices of $w \in \C$.
\end{proof}

Note that if we forgo the requirement that $0$ be a regular value of $f_0$, the choice $g(z) \equiv a_0^{(m)}$ in the above proof suffices, and the degree of $f_0$ is reduced to $m$.

\section{Relations In The Calkin Algebra} \label{Relations In The Calkin Algebra}

In \cite{LRC}, Kriete-Moorhouse investigate compactness of linear combinations of composition operators where the inducing maps lie in the class $\A$. We review some definitions and results from \cite{LRC}, then apply our results from \cref{The CF Problem} to obtain a decomposition of such a composition operator, modulo the ideal $\K$ of compact operators, into a sum of composition operators induced by basic or rational functions (see \cref{definition of basic functions}).

Additionally, we review a result from \cite{LRC} regarding weighted composition operators modulo $\K$ and use this result in the proof of a similar result concerning weighted adjoints of composition operators.

\subsection{Linear Relations in the Calkin Algebra for Composition Operators}
$ $

In \cite{LRC}, Kriete-Moorhouse show that information relating to compactness of a linear combination of compositions operators $c_1 C_{\p_1} + ... +c_r C_{\p_r}$, where $\p_1,...,\p_r \in \A$, is carried by the behavior of the functions $\p_j$ at their points of contact with the unit circle. More precisely, the relevant information for $\p$ at a point of contact $\zeta$ is $D_n(\p, \zeta) = (\p(\zeta), \p'(\zeta)...,\p^{(n)}(\zeta))$, where $n$ is the order of contact of $\p$ with the unit circle at $\zeta$. The following result determines compactness of a linear combination of composition operators for operators induced by functions in $\A$.

\begin{thm}\label{first relation in K}\cite[Theorem~5.13]{LRC} Let $\p_1, . . . , \p_r$ in $\A$ and write $F$ for the union $F(\p_1) \cup ... \cup F(\p_r)$, a finite set. For $\zeta$ in $F$ and $k = 2, 4, 6, . . . ,$ let
\[ \N_k(\zeta) = \{j : \text{ $F(\p_j)$ contains $\zeta$ and $k$ is the order of contact of $\p_j$ at $\zeta$} \} \]
and let
\[ \mathcal{E}_k(\zeta) = \{ D_k(\p_j, \zeta) : \text{$j$ is in $\N_k(\zeta)$} \} .\]
Given complex numbers $c_1, . . . , c_r$, the following are equivalent:
\begin{enumerate}
\item $c_1C_{\p_1} + · · · + c_nC_{\p_r}$ is compact;
\item $\displaystyle \sum_{\substack{j\in \N_k(\zeta) \\ D_k(\p_j ,\zeta)=\textbf{d} }} c_j = 0$ for every $\zeta$ in $F$, every even $k \geq 2$ and every $\textbf{d}$ in $\mathcal{E}_k(\zeta)$.
\end{enumerate}
\end{thm}

Our goal is to decompose $C_\p$, modulo $\K$, into a sum of composition operators induced by basic functions.

\begin{definition} \label{definition of basic functions} A function $\p$ analytic on $\D$ is a basic function with contact at $\zeta$ if the following hold.
\begin{enumerate}
\item $\p$ is a rational function mapping the unit disk $\D$ into itself.
\item $\p(\zeta)$ is on the unit circle, and $\p$ maps the rest of the unit circle into $\D$.
\item $\p(\zeta)$ is a regular value for $\p$.
\end{enumerate}
\end{definition}

\begin{lem} \label{basic building blocks} Let $\p \in \A$ and $\zeta \in F(\p)$. Let $n=2m$ denote the order of contact of $\p$ with $\T$ at $\zeta$. Then there exists a degree $m+1$ basic function $\p_0$ with order of contact $n$ with $\T$ at $\zeta$ which satisfies $D_{n}(\p_0 ,\zeta) = D_{n}(\p ,\zeta)$.
\end{lem}
\begin{proof} We begin by defining $\lambda = \p(\zeta)$ and $f = \tau_{\lambda} \circ \p \circ \tau^{-1}_{\zeta}$. Then $f \in \Pick$ and since $\p \in \A$ has contact with $\T$ of order $n$ at $\zeta$, there exist $a_1, a_2, ..., a_{n-1} \in \R$ and $a_n \in \U$ such that $f$ satisfies
\[ f (z) = 0 + a_1 (z-x) + ... +a_n (z-x)^n +o(|z-x|^n)\]
as $z \to 0$ unrestrictedly in $\U$. By \cref{existence of basic functions}, there exists a degree $m+1$ rational function $F \in \Pz$ that maps $\RH \setminus \{0\}$ into $\U$, has initial Taylor coefficients $a_0=0, a_1, ..., a_n$ at $z=0$, and has $z=0$ as a regular value.

We define $\p_0 = \tau_{\lambda}^{-1} \circ F \circ \tau_{\zeta}$ and get that $\p_0$ is a degree $m+1$ basic function with order contact $n$ with $\T$ at $\zeta$. Note that as a consequence of \fdbf ,~ $D_{n}(F, 0) = D_{n}(f, 0)$ implies that $\p_0$ satisfies $D_{n}(\p_0 ,\zeta) = D_{n}(\p ,\zeta)$ as desired.
\end{proof}

\begin{thm}\label{sum decomposition into basics}  Let $\p \in \A$ with $F(\p) = \{ \zeta_1, ..., \zeta_r \}$ and let $n_j = 2 m_j$ denote the order of contact of $\p$ with the unit circle at $\zeta_j$. Then there exists a decomposition,
\[ C_{\p} \equiv C_{\p_1} + ... + C_{\p_r} \modK, \]
where for each $j=1,...,r$, $\p_j$ is a basic function of degree $m_j+1$ which has contact of order $n_j$ at $\zeta_j$ and satisfies $D_{n_j}(\p_j, \zeta_j) = D_{n_j}(\p, \zeta_j)$.
\end{thm}
\begin{proof}
Existence of $\p_1,...,\p_r$ follows from \cref{basic building blocks}. The result follows by applying \cref{first relation in K} to $\p, \p_1, ... \p_n$ with constants $1, -1,...,-1$.
\end{proof}

Note that we can reduce the degree of the rational maps $\p_1, ..., \p_r$ in the above decomposition to $m_1, ..., m_r$ respectively by relaxing the condition that these functions be basic and requiring only that they be rational self-maps of $\D$ having contact with $\T$ at exactly one point.

\subsection{Weighted Composition Operators and Adjoints in the Calkin Algebra}
$ $

Given a bounded measurable function $w$ on $\T$, we consider the multiplication operator $M_w \colon f \mapsto wf$ which can be viewed as mapping $L^2$ to $L^2$, $H^2$ to $L^2$, or if $w$ is in $H^\infty$, $H^2$ to $H^2$. Additionally, we consider the Toeplitz operator $\displaystyle T_w = \left.PM_w\right|_{H^2}$ (where $P$ is the orthogonal projection of $L^2$ onto $H^2$) which maps $H^2$ to $H^2$ regardless of the choice of $w$.

In \cite{LRC}, it is shown that the coset of the weighted composition operator $M_w C_\p$ modulo the subspace of compact operators from $H^2$ to $L^2$ (also denoted by $\K$ here for convenience) is in some sense determined by the values of $w$ on $E(\p)$. For the special case that $\p \in \A$ has contact with $\T$ at exactly one point, we have the following corollary.

\begin{cor}\label{M_w C_p formula} Let $\p \in \A$ be such that $F(\p) = \{\zeta\}$. Suppose $w$ is a bounded measurable function on $\T$ such that $w$ is continuous at $\zeta$. Then the $H^2$ to $L^2$ operator $M_w C_{\p}$ satisfies
\[ M_w C_{\p} \equiv w(\zeta)C_{\p} \modK. \]
\end{cor}
\begin{proof}
Let $v(z) := w(z) - w(\zeta)$. Then $v$ is bounded on $\T$, continuous at $\zeta$ and satisfies $v \equiv 0$ on $E(\p) = F(\p) = \{ \zeta \}$. Thus by \cite[Theorem~3.1]{LRC}, $M_v C_{\p} = M_w C_{\p} - w(\zeta)C_{\p}$ is compact.
\end{proof}

In \cref{M_w C_p star formula} we prove a similar result for weighted adjoints of composition operators. The proof relies on existence of an $H^{\infty}$ function which satisfies several boundary conditions.

\begin{lem}\label{technical lemma} Let $I$ be an open arc in $\T$ and suppose $\lambda \in I$. Let $v$ be a non-negative bounded function on $\T$ which is continuous on $I$, continuously differentiable on $I \setminus \{ \lambda \}$ and satisfies $v(\lambda)=0$. Then there exists an analytic function $b$ on $\D$ which extends continuously to $\ClD$ and satisfies both $b(\lambda) = 0$ and $|b(e^{i\theta})| \geq v(e^{i\theta})$ on $\T$.
\end{lem}
\begin{proof}
Let $J$ be a closed sub-interval of $I$ whose interior contains $\lambda$, and define $u$ on $J$ by $u(e^{i \theta}) = v(e^{i \theta}) + |e^{i \theta} - \lambda|$ for  $e^{i \theta} \in J$. We extend $u$ to all of $\T$ in such a way that $u$ is continuously differentiable on $\T \setminus \{ \lambda \}$ and satisfies
\[ u(e^{i \theta}) \geq v(e^{i \theta}) + |e^{i \theta} - \lambda| \]
for all $e^{i \theta} \in \T$.
Then $\log u$ is integrable on $\T$ and so we can define an analytic function $h$ on $\D$ by the Herglotz integral
\[ h(z) = \int_0^{2\pi} \frac{e^{i\theta} + z}{e^{i\theta} - z} \log u(e^{i\theta})\frac{d\theta}{2\pi}. \]
Note that since $\log u$ is continuously differential on $\T \setminus \{ \lambda \}$, $h$ extends continuously to $\T \setminus \{ \lambda \}$, see \cite[pp~78-80]{MR0133008}.

Let $b = e^h$, so that $b$ is a bounded analytic function on $\D$ that extends continuously to $\T \setminus \{ \lambda \}$. Moreover,
\[ |b(z)| = \exp \left( \int_0^{2\pi} P_z(e^{i\theta}) \log u(e^{i\theta})\frac{d\theta}{2\pi} \right), \]
where $P_z(e^{i\theta})$ is the Poisson kernel at $z$. Since $\log u(e^{i\theta}) \to -\infty$ as $e^{i\theta} \to \lambda$, standard estimates on $P_z$ show that $b(z) \to 0$ as $z \to \lambda$. Thus $b$ extends continuously to all of $\T$ with $b(\lambda) = 0$. Finally, we note that $|b(e^{i\theta})| = |u(e^{i\theta})| \geq v(e^{i\theta})$ on $\T$.

\end{proof}

\begin{prop}\label{M_w C_p star formula} Let $\p \in \A$ be such that $F(\p) = \{\zeta\}$ and denote $\lambda = \p(\zeta)$. Suppose that $w$ is a bounded measurable function on $\T$ such that $w$ is continuous on $I$ and continuously differentiable on $I \setminus \{ \lambda \}$ for some open arc $I$ in $\T$ containing $\lambda$. Then
\[ M_w C_{\p}^* \equiv w(\lambda)C_{\p}^* \modK, \]
where $M_w$ is viewed as an operator from $H^2$ to $L^2$, and in particular,
\[ T_w C_{\p}^* \equiv w(\lambda)C_{\p}^* \modK. \]
\end{prop}
\begin{proof}
It suffices to prove that $M_v C_{\p}^*$ is compact where $v = w - w(\lambda)$ since
\[ M_w C_{\p}^* = (M_{w-w(\lambda)} + M_{w(\lambda)})C_{\p}^* = M_{v}C_{\p}^* +w(\lambda)C_{\p}^*. \]
Note that $v$ is continuous on $I$ and continuously differentiable on $I \setminus \{ \lambda \}$ and satisfies $v(\lambda) = 0$, and so by \cref{technical lemma} there exists a $b \in H^{\infty}$ which extends continuously to $\T$ and satisfies $b(\lambda) = 0$ and $|b(e^{i\theta})| \geq v(e^{i\theta})$ on $\T$. We get that for all $f \in H^2$,
\begin{align*}
\|M_v C_{\p}^*f\|^2 &= \int_0^{2\pi} |v(e^{i\theta})|^2 |(C_{\p}^*f)(e^{i\theta})|^2 \frac{d\theta}{2\pi} \\
&\leq \int_0^{2\pi} |\ol{b(e^{i\theta})}|^2 |(C_{\p}^*f)(e^{i\theta})|^2 \frac{d\theta}{2\pi} = \|M_{\ol{b}} C_{\p}^*f\|^2,
\end{align*}
and so it suffices to show that $M_{\ol{b}} C_{\p}^*$ is a compact operator from $H^2$ to $L^2$. We write
\[ M_{\ol{b}} C_{\p}^* = PM_{\ol{b}} C_{\p}^* +(I-P)M_{\ol{b}} C_{\p}^* \]
and show that both terms on the right hand side are compact.

We first note that since $\ol{b}$ is a continuous function on $\T$, the $L^2$ operator ${(I-P)M_{\ol{b}}P}$ is compact (see the version of Hartman's theorem in \cite[p~214, Theorem~2.2.5]{MR1864396}). Thus, the term $(I-P)M_{\ol{b}} C_{\p}^* = (I-P)M_{\ol{b}}P C_{\p}^*$ is compact from $H^2$ to $L^2$.

We now show $PM_{\ol{b}} C_{\p}^* = T_{\ol{b}} C_{\p}^*$ is compact on $H^2$ by looking at its adjoint $(T_{\ol{b}} C_{\p}^*)^* = C_{\p} T_b$. Since $b \in H^{\infty}$, we have $C_{\p} T_b = C_{\p} M_b = M_{b \circ \p} C_{\p}$. We wish to apply \cref{M_w C_p formula} to $M_{b \circ \p} C_{\p}$. Recall that $\p \in \A$ with $F(\p) = \{ \zeta \}$, and note that although the non-tangential boundary function $\p(e^{i\theta})$ is in general defined only almost everywhere, we can extend it to all of $\T$ by setting $\p(e^{i\theta}) = \p(\zeta)$ on the remaining set of measure zero. It follows from the definition of $\A$ that this extension (which we also call $\p$) is continuous at $\zeta$. Thus $b \circ \p$ is continuous at $\zeta$ and $b(\p(\zeta)) = b(\lambda) = 0$, and so \cref{M_w C_p formula} implies that $M_{b \circ \p} C_{\p}$ is compact as desired.
\end{proof}

\section{Adjoint Formula for Rationally Induced Composition Operators} \label{Adjoint Formula for Rational Case}

Recent work of Cowen-Gallardo \cite{MR2253727}, Hammond-Moorhouse-Robbins \cite{MR2394110} and Bourdon-Shapiro \cite{BS} has produced pointwise formulas for $C_\p^{*}$, where the inducing map $\p$ is rational. The constituent parts of these pointwise formulas contain multiple-valued analytic functions which do not necessarily represent well-defined operators individually. We show how to work with these pointwise formulas to produce legitimate operator equations involving $C_\p^{*}$ for the rational case. We then consider the case where $\p$ is basic and reduce our equations to the Calkin algebra.

\subsection{From Pointwise Formula To Operator Equation} \label{rational inducing map}
$ $

Let $\p$ be a rational self-map of $\D$ of degree $d$. We associate with $\p$ its exterior map $\p_e := \rho \circ \p \circ \rho$, where $\rho \colon z \to 1/\ol{z}$ is the inversion in the unit circle. Then $\p_e$ maps $\D_e := \{ z \in \C \colon |z|>1 \}$ into itself, and so $\p_e^{-1}(\D) \subset \D$.

For any simply connected domain $V$ consisting of regular values of $\p_e$, there exist $d$ distinct branches $\s_1,...,\s_d$ of $\p_e^{-1}$ defined on $V$, and we have that the sets $\s_1(V), ..., \s_d(V)$ are pairwise disjoint (see \cite{BS}). 
Note that one possible choice of $V$ is the unit disk with radial slits from each critical value of $\p_e$ to the unit circle removed. A choice that may be much smaller but sufficient for our needs is a small neighborhood of a regular value of $\p_e$.

We use the following variant of the pointwise formula for $C_\p^*$ introduced by Bourdon-Shapiro in \cite{BS}.

\begin{prop}\label{pointwise adjoint formula}\cite[Corollary~8]{BS} Suppose that $V$ is a set on which $d$ distinct branches $\s_1,...,\s_d$ of $\p_e^{-1}$ are defined. Then for all $f \in H^2$ and all $z \in V \cap \D$,
\begin{align}\label{pointwise adjoint formula - the formula}
C_\p^*f(z) = \frac{f(0)}{1-\ol{\p(0)}z} +\sum_{j=1}^{d} z \s'_j(z)S^*f(\s_j(z)),
\end{align}
where $S^*$ is the adjoint of the shift operator $S$ defined by $(S f)(z) = zf(z)$.
\end{prop}

Note that \cref{pointwise adjoint formula - the formula} can be rewritten, at least formally, as
\begin{align} \label{wishful thinking formula}
C_\p^* = \Lambda + \sum_{j=1}^{d} M_{h_j}C_{\s_j}S^*,
\end{align}
where $h_j(z) = z \s_j'(z)$, and $\Lambda$ is the rank one operator defined by $\Lambda(f) := \frac{f(0)}{1-\ol{\p(0)}z}$. However, the maps $\s_j$ are not in general analytic on all of $\D$, and so the operators on the right hand side are not in general ``legitimate'' operators.

Bourdon-Shapiro define $\p$ as outer regular when its critical values all lie in $\D$. Note that for outer regular functions $\p$ we can choose $V = r\D$ for some $r>1$. Then, restricting domains to $\D$, we have that $\s_1,...,\s_d$ are analytic self-maps of $\D$ and $h_1,...,h_d$ are $H^{\infty}$ functions. Thus, equation \cref{wishful thinking formula} is a legitimate operator equation in the outer regular case \cite[Theorem~13(a)]{BS}. Unfortunately, the outer regular case is only possible for order of contact $2$ functions.

\begin{prop} If $\p$ is a rational self-map of $\D$ having order of contact $n>2$ with the unit circle at $\zeta$, then $\p$ is not outer regular.
\end{prop}
\begin{proof} Let $\p$ be as in the assumption and in order to obtain a contradiction suppose that $\p$ is outer regular. Then there exists a branch $\s$ of $\p_e^{-1}$ mapping $\p(\zeta)$ to $\zeta$ and defined on all of $\D$. We transfer the maps to the upper half-plane and work with
$\pt = \tau_{\p(\zeta)}^{-1} \circ \p \circ \tau_\zeta$,
$\pet = \tau_{\p(\zeta)}^{-1} \circ \p_e \circ \tau_\zeta$ and
$\st = \tau_{\zeta}^{-1} \circ \s \circ \tau_{\p(\zeta)}$, noting that all three maps fix $0$ and $\pt, \st \in \Pz$.

Let $a_0, a_1, a_2, a_3$ denote the initial Taylor coefficients of $\pt$ at $z=0$, and $b_0,b_1,b_2,b_3$ denote those of $\st$.
In \cref{order of contact of sigma}, we show that $\s$ has order of contact $n$ at $\p(\zeta)$, and so by \cref{CFP condition} the Hankel matrices $H_m(a_1,...,a_{n-1})$ and $H_m(b_1,...,b_{n-1})$ are both positive, and in particular, their $2$nd leading principal minors, $a_1 a_3 -a_2^2$ and $b_1 b_3 -b_2^2$, are positive.

In \cref{derivatives of p and pe} we show that $\pt$ and $\pet$ have equal Taylor coefficients ${a_0,...,a_{n-1}}$. Noting that $\pet \circ \st = id$, we can therefore express $b_1,b_2,b_3$ in terms of $a_1,a_2,a_3$ using \fdbf. We reach a contradiction by calculating that $b_1b_3-b_2^2 = \frac{-(a_1 a_3 - a_2^2 )}{a_1^6} < 0$.
\end{proof}

\subsection{Generalized Adjoint Formula in the Calkin Algebra} \label{Adjoint Formula for Basic Case}
$ $

In general, the set $V$ on which $\s_1,...,\s_d$ are analytic can not be chosen to contain all of $\D$, and so the formal operators $C_{\s_1}, ..., C_{\s_d}$ in \cref{wishful thinking formula} are not legitimate operators. We overcome this difficulty by pre-composing with a map $\psi$ with image contained in $V \cap \D$ to obtain analytic self-maps of $\D$, $\s_1 \circ \psi,..., \s_d \circ \psi$. This enables us to write a legitimate operator formula for $C_{\psi}C_{\p}^*$.

\begin{prop}\label{adjoint operator formula} Let $\psi$ be an analytic self-map of $\D$ satisfying $\ol{\psi(\D)} \subset V$. Then
\[ C_{\psi}C_\p^* = C_{\psi}\Lambda + \sum_{j=1}^{d} M_{h_j \circ \psi}C_{\s_j \circ \psi}S^* \]
where $h_j(z) = z \s'_j(z)$, $S^*$ is the adjoint of the shift operator and $\Lambda$ is the rank one operator defined by $ \Lambda(f) := \frac{f(0)}{1-\ol{\p(0)}z}$.

In particular, $h_j \circ \psi$ are $H^{\infty}$ functions and $\s_j \circ \psi$ are analytic self-maps of $\D$.
\end{prop}
\begin{proof} Since $\psi$ maps $\D$ into $V$ and $\s_1,...,\s_d$ are analytic on $V$, we get that $\s_1 \circ \psi ,...,\s_d \circ \psi$ are analytic self-maps of $\D$ (recall that $\p_e^{-1}(\D) \subset \D$). The functions $\s'_1,...,\s'_d$ are analytic on $V$ and so bounded on $\psi(\D)$, so $h_1 \circ \psi, ..., h_d \circ \psi$ are $H^{\infty}$ functions. To complete the proof, note that by the pointwise formula given in \cref{pointwise adjoint formula}, for all $f \in H^2$ and all $z \in \D$ we have
\begin{align*}
C_{\psi}C_\p^*f(z) = C_\p^*f(\psi(z)) &= \frac{f(0)}{1-\ol{\p(0)}\psi(z)} +\sum_{j=1}^{d} \psi(z) \s'_j(\psi(z))S^*f(\s_j(\psi(z))) \\
&= (C_{\psi}\Lambda f)(z) + \sum_{j=1}^{d} (M_{h_j \circ \psi}C_{\s_j \circ \psi}S^*f)(z).
\end{align*}
\end{proof}

We turn to the case where $\p$ is a basic function of degree $d$ with order of contact $n$ at $\zeta$. We denote $\lambda = \p(\zeta)$ and note that $\rho(\lambda) = \lambda$ is a regular value for $\p_e$, and so there exists a neighborhood $V(\p)$ of $\lambda$ consisting of regular values of $\p_e$. Recall that $\s_1(V(\p)),...,\s_d(V(\p))$ are pairwise disjoint, and let $\s$ denote the unique branch of $\p_e^{-1}$ that maps $\lambda$ to $\zeta$. The following is a generalization of \cite[Corollary~15]{BS}.

\begin{lem}\label{adjoint formula mod compacts 0.5} Let $\p$ be a basic function with contact with $\T$ at $\zeta$, and let $\psi$ be an analytic self map of $\D$ satisfying $\ol{\psi(\D)} \subset V(\p)$. Then 
\[C_\psi C_\p^* \equiv M_{h \circ \psi} C_{\s \circ \psi} S^* \modK \]
where $h(z) := z \s'(z)$.

In particular, $h \circ \psi$ is an $H^{\infty}(\D)$ function and $\s \circ \psi$ is an analytic self-map of $\D$.
\end{lem}
\begin{proof} 

Notice that since $\rho$ is the identity on $\T$, $\p_e$ maps exactly one point in $\T$ (the point $\zeta$) into $\T$ and so $\p_e^{-1}(\T) = \{ \zeta \}$. We also have that the branches $\s_1,...,\s_d$ map $\ol{\psi(\D)}$ to pairwise disjoint closed subsets of $\ClD$. Thus, for $\s_j \neq \s$ the closed set $\s_j(\ol{\psi(\D)})$ does not intersect $\T$ and so $\| \s_j \circ \psi \|_\infty < 1$ and we get that the composition operator $C_{\s_j \circ \psi}$ is compact.

Noting that $\Lambda \colon f \to \frac{f(0)}{1-\ol{\p(0)}}z$ is rank one, reducing \cref{adjoint operator formula} modulo the compacts gives
\[ C_{\psi}C_\p^* \equiv M_{h \circ \psi}C_{\s \circ \psi}S^* \modK \]
where $h$ is defined by $h(z) = z \s'(z)$, and by the same proposition $h$ satisfies $h \circ \psi$ is an $H^{\infty}(\D)$ function and $\s \circ \psi$ is an analytic self-map of $\D$.
\end{proof}

For a basic function $\p$, there exists a neighborhood $W(\p)$ of $\lambda$ such that $\ol{W(\p)} \subset V(\p)$, and on which $\s$ is bounded away from zero (recall that $\s(\lambda) = \zeta$).
Restricting to this neighborhood enables the removal of $S^*$ from the formula. In essence, working modulo $\K$, we transfer a term originating from a summand in the variant formula (\cref{pointwise adjoint formula}), to a term resembling a summand in the original formula \cite[Theorem~7]{MR2394110}.

\begin{lem}\label{adjoint formula mod compacts 1} Let $\p$ be a basic function with contact with $\T$ at $\zeta$, and $\psi$ be an analytic self map of $\D$ satisfying $\psi(\D) \subset W(\p)$. Then 
\[C_\psi C_\p^* \equiv M_{g \circ \psi} C_{\s \circ \psi} \modK \]
where $g(z) := \frac{z \s'(z)}{\s(z)}$.

In particular, $g \circ \psi$ is an $H^{\infty}(\D)$ function and $\s \circ \psi$ is an analytic self-map of $\D$.
\end{lem}
\begin{proof} By \cref{adjoint formula mod compacts 0.5}, there exists a $K_1 \in \K$ such that
\[ C_{\psi}C_\p^* = M_{h \circ \psi}C_{\s \circ \psi}S^* + K_1, \]
where $h$ is defined by $h(z) = z \s'(z)$ and satisfies $h \circ \psi \in H^{\infty}(\D)$, and $\s \circ \psi$ is an analytic self-map of $\D$. Note that since $\psi$ maps $\D$ into $W(\p)$, the map $\s \circ \psi$ is bounded away from zero on $\D$, and so $g \circ \psi$ is an $H^{\infty}(\D)$ function. Recall that for any $f \in H^2$ and any nonzero $z \in \D$, we have that $(S^* f)(z) = \dfrac{f(z) - f(0)}{z}$, and so for any $f \in H^2$ and $z \in \D$ we have 
\begin{align*}
(C_{\psi}C_\p^*f)(z) &=(M_{h \circ \psi}C_{\s \circ \psi}S^*f)(z) + (K_1f)(z) \\
&= \psi(z)\s'(\psi(z)) \frac{f(\s(\psi(z)))-f(0)}{\s(\psi(z))} + (K_1f)(z) \\
&= g(\psi(z))f(\s(\psi(z))) - g(\psi(z))f(0) + (K_1f)(z).
\end{align*}
Thus $C_{\psi}C_\p^* = M_{g \circ \psi} C_{\s \circ \psi} - K_2 + K_1$, where $K_2$ is the rank one operator defined by $K_2 \colon f \mapsto f(0) \cdot g \circ \psi$, and the proof is complete.
\end{proof}

For a linear fractional map $\p(z) = \dfrac{az+b}{cz+d}$, we have that $\p_e(z)$ is invertible, and $\p_e^{-1}(z) = \s(z)$ is the Krein adjoint of $\p$. In \cite{LFADJ}, Kriete-MacCluer-Moorhouse developed the adjoint formula modulo $\K$ for this case, which states
\[ C_\p^* \equiv \frac{1}{|\p'(\zeta)|} C_\s \modK. \]
This can easily be extended to $\tau \in \A$ with $F(\tau) = \{ \zeta \}$, to produce a formula for $C_\tau^*$ provided $\tau$ has order of contact $2$ with $\T$ at $\zeta$. We can now generalize this adjoint formula to higher orders of contact.

For what follows, recall that $f(\xi)$ denotes the non-tangential limit of $f$ at $\xi$ for a function $f$ of $\D$ and $\xi \in \T$.

\begin{prop}\label{super nice adjoint formula mod compacts} Let $\p$ be a basic function with contact with $\T$ at $\zeta$ and let $\psi$ be a self-map of $\D$ satisfying for some $\eta \in \T$:
\begin{enumerate}
\item \label{one}$\psi(\D) \subset W(\p)$;
\item \label{two}$\psi$ is analytic at $\eta$ and $\psi(\eta) = \lambda$;
\item \label{three} $\psi^{-1}( \{ \lambda \} ) :=  \{ \beta \in \T \colon \psi(\beta) \text{ exists and is equal to } \lambda \} = \{ \eta \}$.
\end{enumerate}
Then we have that the map $\s \circ \psi$ is in the class $\A$ with $F(\s \circ \psi) = \{ \eta \}$ and that
\[C_{\psi} C_\p^* \equiv \frac{1}{|\p'(\zeta)|}C_{\s \circ \psi} \modK. \]
\end{prop}
\begin{proof}
By \cref{adjoint formula mod compacts 1} we have that
\begin{align} \label{from adjoint formula}
C_\psi C_\p^* \equiv M_{g \circ \psi}C_{\s \circ \psi} \modK,
\end{align}
where $g(z) = \frac{z \s'(z)}{\s(z)}$, $g \circ \psi$ is an $H^{\infty}(\D)$ function and $\s \circ \psi$ is an analytic self-map of $\D$. Note that it follows from \ref{one} and \ref{two} and the definition of $W(\p)$ that the maps $\s \circ \psi$ and $g \circ \psi$ are both analytic at $\eta$. In order to apply \cref{M_w C_p formula} to $M_{g \circ \psi}C_{\s \circ \psi}$, it remains to show that $\s \circ \psi \in \A$.

Recall that $E(\s \circ \psi) = \ol{ \bigcup_{|\alpha|=1} spt(\mu^s_\alpha) }$, 
where $\mu^s_\alpha$ is the singular part of the Clark measure for $\s \circ \psi$ which is carried by
\[ (\s \circ \psi)^{-1}(\{\alpha\}) = \{ \beta \in \T \colon (\s \circ \psi)(\beta) \text{ exists and is equal to } \alpha \}. \]
Suppose $(\s \circ \psi)(\beta) = \alpha$ for some $\alpha, \beta \in \T$. Then applying $\p_e$ we get that $\psi(\beta)$ exists and $\s(\psi(\beta)) = \alpha$. Recall that $\p_e^{-1}(\T) = \{ \zeta \}$, $\s(\lambda) = \zeta$ and $\s$ is univalent on $V(\p)$. Thus $\s(\psi(\beta)) = \alpha$ implies that $\alpha = \zeta$ and $\psi(\beta) = \lambda$, and so by \ref{three} we have $\beta = \eta$. Therefore $spt(\mu^s_\alpha)$ is empty for $\alpha \neq \zeta$, and $spt(\mu^s_\zeta) \subset \{ \eta \}$. We conclude that $E(\s \circ \psi) \subset \{ \eta \}$ is finite and so $E(\s \circ \psi) = F(\s \circ \psi) = \{ \eta \}$. Since $\s \circ \psi$ is analytic at $\eta$ and does not map an arc of $\T$ containing $\eta$ into $\T$, $\s \circ \psi$ has finite order of contact at $\eta$ and a Taylor expansion to that order about $\eta$.
Thus $\s \circ \psi$ is in the class $\A$.

Now applying \cref{M_w C_p formula} to $M_{g \circ \psi}C_{\s \circ \psi}$, \cref{from adjoint formula} becomes
\[C_\psi C_\p^* \equiv (g \circ \psi)(\lambda)C_{\s \circ \psi} \modK.\]
In order to complete the proof, we calculate $(g \circ \psi)(\eta) = g(\lambda) = \frac{\lambda}{\zeta \p'(\zeta)}$ and note that $|\p'(\zeta)| = \zeta \ol{\lambda} \p'(\zeta)$ by the Julia Carath\' eodory Theorem \cite{MR1397026}.
\end{proof}

\subsection{Relationship Between $\p$ and $\s$} \label{section almost inverse} 
$ $

Let $\p$ be a rational function with order of contact $n$ with the unit circle at $\zeta$, mapping $\zeta$ to $\lambda$, and suppose that $\s$ is a branch of $\p_e^{-1}$ defined on some neighborhood of $\lambda$ and mapping $\lambda$ to $\zeta$. Then $\s \circ \p_e = id$ near $\zeta$ and ${\p_e \circ \s = id}$ near $\lambda$, so that
$D_n(\s \circ \p_e, \zeta) = (\zeta, 1, 0,...,0)$ and $D_n(\p_e \circ \s, \lambda) = (\lambda, 1, 0,...,0)$.
Although $\p$ and $\s$ are not inverse functions, we show that they are ``almost inverse'' in the sense that $D_{n-1}(\s \circ \p, \zeta) = (\zeta, 1, 0,...,0)$, and $D_{n-1}(\p \circ \s, \lambda) = (\lambda, 1, 0,...,0)$.
However $D_{n}(\s \circ \p, \zeta)$ and $D_{n}(\p \circ \s, \lambda)$ are not generally equal, and the precise way that they can differ is one key to our main result.

Throughout this section, we use \fdbf. We transfer $\D$ to $\U$ using the family of conformal maps $\tau_\alpha \colon \D \to \U$ for $\alpha \in \T$, defined by
$ \tau_\alpha \colon z \mapsto i\frac{\alpha-z}{\alpha+z}$, and analyze the relationships of
$\pt = \tau_{\lambda} \circ \p \circ \tau_{\zeta}^{-1}$,
$\pet = \tau_{\lambda} \circ \p_e \circ \tau_{\zeta}^{-1}$ and
$\st = \tau_{\zeta} \circ \s \circ \tau_{\lambda}^{-1}$.
Note that $\pt, \st \in \Pz$ and that $\pet$ is the upper half-plane exterior map associated with $\pt$, that is, $\pet(z) = \ol{ \pt (\ol{z})}$.

\begin{lem}\label{derivatives of p and pe} $D_{n-1}(\p, \zeta) = D_{n-1}(\p_e, \zeta)$ and $\p^{(n)}(\zeta) \neq \p_e^{(n)}(\zeta)$.
\end{lem}
\begin{proof} By the discussion in \cref{Fa di Bruno section}, it suffices to show that 
\[ D_{n-1}(\pt, 0) = D_{n-1}(\pet, 0) \quad \text{and} \quad \pt^{(n)}(0) \neq \pet^{(n)}(0). \]
Let $\pt(z) = \sum_0^{\infty} a_k z^k$ be the Taylor expansion of $\pt$ about $0$. Then $\pet(z) = \ol{\pt(\ol{z})} = \sum_0^{\infty}\ol{a_k}z^k $, and so we have
\[ \pt^{(k)}(0) = k! a_k \quad \text{and} \quad \pet^{(k)}(0) = k! \ol{a_k} \]
for all $k$. To complete the proof, recall that $\p$ has order of contact $n$ with the unit circle, and so $a_0, a_1,..., a_{n-1} \in \R$ and $\im a_n >0$.
\end{proof}

\begin{prop} \label{s circ p} The $n^{th}$ order data for $\s \circ \p$ and $\p \circ \s$ is given by
\begin{align*}
D_n(\s \circ \p, \zeta) &= \left( \zeta, 1, 0,...,0, \frac{c}{\p'(\zeta)} \right) \\
D_n(\p \circ \s, \lambda) &= \left( \lambda, 1, 0,...,0, \frac{c}{\p'(\zeta)^n} \right),
\end{align*}
where $c$ is the non zero constant given by $c = \p^{(n)}(\zeta) - \p_e^{(n)}(\zeta).$
\end{prop}
\begin{proof}
By \cref{derivatives of p and pe} $D_{n-1}(\p, \zeta) = D_{n-1}(\p_e, \zeta)$, and so by \fdbf , we get
\[ D_{n-1}(\s \circ \p, \zeta) = D_{n-1}(\s \circ \p_e, \zeta) = \left( \zeta, 1, 0,...,0 \right), \]
\[ D_{n-1}(\p \circ \s, \lambda) = D_{n-1}(\p_e \circ \s, \lambda) = \left( \zeta, 1, 0,...,0 \right). \]
We use \cref{useful Faa di Bruno} to we write $(\s \circ \p)^{(n)}(\zeta)$ and $(\s \circ \p_e)^{(n)}(\zeta)$ as:
\begin{align*}
(\s \circ \p)^{(n)}(\zeta) &= \s^{(n)}(\lambda)\p'(\zeta)^n + F(D_{n-1}(\s, \lambda), D_{n-1}(\p,\zeta)) + \s'(\lambda)\p^{(n)}(\zeta), \\
(\s \circ \p_e)^{(n)}(\zeta) &= \s^{(n)}(\lambda)\p_e'(\zeta)^n + F(D_{n-1}(\s, \lambda), D_{n-1}(\p_e,\zeta)) + \s'(\lambda)\p_e^{(n)}(\zeta). 
\end{align*}
Now note that $(\s \circ \p_e)^{(n)}(\zeta) =0$ and $D_{n-1}(\p, \zeta) = D_{n-1}(\p_e, \zeta)$, so that subtracting the above equations yields
\[(\s \circ \p)^{(n)}(\zeta) = \s'(\lambda) \cdot (\p^{(n)}(\zeta) - \p_e^{(n)}(\zeta)). \]
To complete the proof of the first statement, note that \cref{derivatives of p and pe} implies that $\s'(\lambda) = 1/ \p_e'(\zeta) = 1/ \p'(\zeta)$ and that $c = \p^{(n)}(\zeta) - \p_e^{(n)}(\zeta)$ is non zero.

Similarly, we apply \cref{useful Faa di Bruno} to $(\p \circ \s)^{(n)}(\lambda)$ and $(\p_e \circ \s)^{(n)}(\lambda)$, to get
\begin{align*}
(\p \circ \s)^{(n)}(\zeta) &= \p^{(n)}(\zeta)\s'(\lambda)^n + F(D_{n-1}(\p, \zeta), D_{n-1}(\s,\lambda)) + \p'(\zeta)\s^{(n)}(\lambda), \\
(\p_e \circ \s)^{(n)}(\zeta) &= \p_e^{(n)}(\zeta)\s'(\lambda)^n + F(D_{n-1}(\p_e, \zeta), D_{n-1}(\s,\lambda)) + \p_e'(\zeta)\s^{(n)}(\lambda).
\end{align*}
We subtract the above equations and substitute $1/ \p'(\zeta)$ for $\s'(\lambda)$ to get
\[(\p \circ \s)^{(n)}(\zeta) = 1/ \p'(\zeta)^n \cdot (\p^{(n)}(\zeta) - \p_e^{(n)}(\zeta)). \]
\end{proof}

As an additional application of Fa\`a di Bruno's Formula, we calculate the order of contact of $\s$ with $\T$ at $\lambda$.

\begin{prop}\label{order of contact of sigma} The map $\s$ has order of contact $n$ with $\T$ at $\lambda$.
\end{prop}
\begin{proof} Let $\pt(z) = \sum_0^{\infty} a_k z^k$ and $\st(z) = \sum_0^{\infty} b_k z^k$ be the Taylor expansions of $\pt$ and $\st$ about $0$ respectively, and recall that $\pet(z) = \ol{\pt(\ol{z})} = \sum_0^{\infty} \ol{a_k} z^k$. Note that since $\p$ has order of contact $n$ with $\D$, we have that $a_0,...,a_{n-1} \in \R$ and $\im a_n >0$.

It suffices to show that $b_0 = 0, b_1, ..., b_{n-1} \in \R$ and $\im b_n >0$. By induction on $k = 2,...,n-1$ and using \fdbf ~for $(\st \circ \pet)^{(k)}(0)$, we see that $b_0,...,b_k$ are real valued. To see that $\im b_n >0$, we use \cref{useful Faa di Bruno} and write
\[ 0 = (\st \circ \pet)^{(n)}(0) = \st^{(n)}(0)\pet'(0)^n + F(D_{n-1}(\st, 0), D_{n-1}(\pet, 0)) + \st'(0)\pet^{(n)}(0).\]
Taking imaginary parts we get $\im (b_n) = \dfrac{b_1}{a_1^n} \im(a_n)$, and since $b_1 = 1/a_1 > 0$ we get that $\im (b_n) > 0$.
\end{proof}

\section{Essential normality} \label{Essential normality}

We begin by characterizing essential normality for the basic case, then prove our main theorem characterizing essential normality for $C_\p$ for all $\p \in \A$. Finally, we construct essentially normal composition operators which have arbitrary even order of contact with the unit circle at one point.

\begin{definition}
For $\epsilon > 0$ we define $\psi_{\lambda, \epsilon}$ to be the Riemann mapping from $\D$ onto $\{ |z - \lambda| < \epsilon \} \cap \D$, which fixes $\lambda$. Note that $\psi = \psi_{\lambda, \epsilon}$ extends continuously to $\T$ and analytically across $\T$ in a neighborhood of $\lambda$.
\end{definition}

\subsection{Essential Normality For Basic Composition Operators}
$ $

\begin{prop}\label{basic ess normality - part 1} Let $\p$ be a basic function with contact at $\zeta$ which fixes $\zeta$. Let $n$ be the order of contact of $\p$ with $\T$ at $\zeta$, and let $\s$ be the unique branch of $\p_e^{-1}$ defined on a neighborhood of $\zeta$ which fixes $\zeta$. There exists $\epsilon > 0$ such that for $\psi = \psi_{\zeta, \epsilon}$, the following conditions are equivalent.
\begin{enumerate}
\item \label{BEN1} $C_\psi[C_\p^{*}, C_\p]$ is compact.
\item \label{BEN2} $D_n (\s \circ \p, \zeta) = D_n (\p \circ \s, \zeta)$.
\item \label{BEN3} $\p'(\zeta) = 1$.
\end{enumerate}
\end{prop}
\begin{proof} \ref{BEN1} $\iff$ \ref{BEN2}: For $W = W(\p)$, note that $\p^{-1}(W) \cap W$ is open and contains $\zeta$, and choose $\epsilon >0$ such that $ \{ |z - \zeta| \leq \epsilon \} \cap \D$ is contained in $\p^{-1}(W) \cap W$. Denote $\psi = \psi_{\zeta, \epsilon}$ and note that both $\psi$ and $\p \circ \psi$ map $\D$ into $W$, are analytic at $\zeta$ and fix $\zeta$ and satisfy $\psi^{-1}( \{ \zeta \} ) = (\p \circ \psi)^{-1}( \{ \zeta \} )= \{ \zeta \}$, and so by \cref{super nice adjoint formula mod compacts} we have
\[ C_{\psi} C_\p^* \equiv \frac{1}{|\p'(\zeta)|}C_{\s \circ \psi} \modK , \qquad
C_{\p \circ \psi} C_\p^* \equiv \frac{1}{|\p'(\zeta)|}C_{\s \circ \p \circ \psi} \modK. \]
Thus, we can express $C_\psi[C_\p^{*}, C_\p]$ in the Calkin algebra by
\begin{align*}
C_\psi[C_\p^{*}, C_\p] \equiv C_\psi C_\p^* C_\p - C_{\p \circ \psi} C_\p^*
\equiv \frac{1}{|\p'(\zeta)|}(C_{\p \circ \s \circ \psi} - C_{\s \circ \p \circ \psi}) \modK.
\end{align*}

Note that $\psi$ maps an arc of $\T$ containing $\zeta$ into $\T$, and recall that by \cref{order of contact of sigma} $\s$ has order of contact $n$ at $\zeta$. Thus, by \cref{contact for composite} both $\s \circ \p \circ \psi$ and $\p \circ \s \circ \psi$ have order of contact $n$ at $\zeta$. Note also that $\s \circ \p \circ \psi$ and $\p \circ \s \circ \psi$ are is in the class $\A$ with $F (\s \circ \p \circ \psi) = F(\p \circ \s \circ \psi) = \{ \zeta \}$. Therefore, by \cref{first relation in K}, $C_{\s \circ \p \circ \psi}-C_{\p \circ \s \circ \psi}$ is compact if and only if $D_n (\s \circ \p \circ \psi, \zeta) = D_n (\p \circ \s \circ \psi, \zeta)$.
Since $\psi$ is invertible in a neighborhood of $\zeta$ and as a consequence of \fdbf, we get that $D_n (\s \circ \p \circ \psi, \zeta) = D_n (\p \circ \s \circ \psi, \zeta)$ if and only if $D_n (\s \circ \p, \zeta) = D_n (\p \circ \s, \zeta)$.

\ref{BEN2} $\iff$ \ref{BEN3}: By \cref{s circ p}, the $n^{th}$ order data for $\s \circ \p$ and $\p \circ \s$ is equal if and only if $\p'(\zeta) = \p'(\zeta)^n$. Note that $\p'(\zeta) > 0$ since $\p$ fixes $\zeta$, and so $\p'(\zeta) = \p'(\zeta)^n$ if and only if $\p'(\zeta) = 1$.
\end{proof}

If $C_\psi$ were bounded below, compactness of $C_\psi A$ would imply compactness of $A$ for any operator $A$. We prove a weaker result using a technique from \cite[section~4]{MR0350487}. We denote the characteristic function on $\T$ for the arc $\Gamma_\delta = \{ |z-\zeta| < \delta \} \cap \T$ by $\chi_\delta$.

\begin{lem} \label{basic ess normality - part 2} Let $\epsilon >0$ and $\psi = \psi_{\zeta, \epsilon}$. Then for any $\chi = \chi_\delta$ with $0< \delta < \epsilon$, and any operator $A$ we have
\[ \text{ $C_\psi A$ is compact $\Longrightarrow$ $T_\chi A$ is compact.} \]
\end{lem}
\begin{proof} Suppose that $C_\psi A$ is compact. It suffices to show that the $H^2$ to $L^2$ operator $M_\chi A$ is compact. Let $t_1, t_2$ be such that $e^{it}$ parametrizes the curve $\Gamma_\delta$ for $t_1 \leq t \leq t_2$, and let $\alpha_1, \alpha_2$ be such that $\psi(e^{it})$ parametrizes the curve $\Gamma_\delta$ for $\alpha_1 \leq t \leq \alpha_2$. Then
\begin{align*}
\| M_{\chi} f \|_{L_2}^{2} &= \int_{t_1}^{t_2} |f(e^{i\theta})|^2 \frac{d\theta}{2\pi}
= \int_{\alpha_1}^{\alpha_2} |f(\psi(e^{it}))|^2 |\psi'(e^{it})| \frac{dt}{2\pi} \\
&\leq m \int_{\alpha_1}^{\alpha_2} |f(\psi(e^{it}))|^2 \frac{dt}{2\pi} \leq m \| C_\psi f \|^2_{L^2},
\end{align*}
where $m = \max \{ \psi'(e^{it}) \colon \alpha_1 \leq t \leq \alpha_2 \}$ is finite since $\psi$ can be analytically extended to a neighborhood of $\Gamma$.

Note that $C_\psi A$ is compact from $H^2$ to $H^2$ and so it is compact from $H^2$ to $L^2$, and that for all $f \in H^2$ we have that $\| M_{\chi} A f \|_{L_2} \leq  \sqrt{m} \| C_\psi A f\|_{L_2}$. Thus $M_{\chi} A$ is compact from $H^2$ to $L^2$.
\end{proof}

\begin{prop} \label{basic essential normality} Suppose $\p$ is a basic function with contact at $\zeta$ which fixes $\zeta$. Then $C_\p$ is essentially normal if and only if $\p'(\zeta)=1$.
\end{prop}
\begin{proof} By \cref{basic ess normality - part 1}, there exists $\epsilon >0$ such that for $\psi = \psi_{\zeta, \epsilon}$, we have $C_\psi[C_\p^*, C_\p]$ is compact if and only if $\p'(\zeta) = 1$. Thus, it suffices to show that compactness of $C_\psi[C_\p^*, C_\p]$ implies compactness of $[C_\p^*, C_\p]$. 

Suppose that $C_\psi[C_\p^*, C_\p]$ is compact. Then by \cref{basic ess normality - part 2} with $\delta = \epsilon / 2$ and $\chi = \chi_\delta$, we get that $T_\chi [C_\p^*, C_\p]$ is compact.
Now since $\p \in \A$ with $F(\p) =\{ \zeta \}$ and $\chi$ is continuously differentiable at $\zeta$, using  \cref{M_w C_p formula} and \cref{M_w C_p star formula} we get
\[ T_\chi [C_\p^*, C_\p]
\equiv \chi(\p(\zeta)) C_\p^* C_\p - \chi(\zeta) C_\p C_\p^*
\equiv [C_\p^*, C_\p] \modK, \]
and so $[C_\p^*, C_\p]$ is compact as well.
\end{proof}

\subsection{General Essential Normality} \label{General Essential Normality}
$ $

In this section we prove our main theorem identifying the essentially normal composition operators induced by a general function in the class $\A$. The first statement in the following lemma is due to Clifford and Zheng \cite{MR1757084} in the case where $\p_1$ and $\p_2$ are linear fractional maps.

\begin{lem}\label{combinations of basic and adjoint} Let $\p_1$ and $\p_2$ be basic functions with contact at $\zeta_1$ and $\zeta_2$ respectively, and denote $\lambda_1 =\p_1(\zeta_1)$ and $\lambda_2 = \p_2(\zeta_2)$. Then the following hold.
\begin{enumerate}
\item If $\zeta_1 \neq \zeta_2$ then $ C_{\p_1}^*C_{\p_2} \equiv 0 \modK$.
\item If $\lambda_1 \neq \lambda_2$ then $ C_{\p_2}C_{\p_1}^* \equiv 0 \modK$. 
\end{enumerate}
\end{lem}
\begin{proof} For the first part, suppose that $\zeta_1 \neq \zeta_2$ and note that $b(z) = \frac{z-\zeta_1}{\zeta_2-\zeta_1}$ is continuous at $\zeta_1$ and $\zeta_2$ and satisfies $b(\zeta_1)=0$ and $b(\zeta_2)=1$. Then by \cref{M_w C_p formula}, $T_b C_{\p_2} \equiv C_{\p_2}$ and $T_{\ol{b}} C_{\p_1} \equiv 0$ (mod $\K$), and so,
\[ C_{\p_1}^*C_{\p_2} \equiv  C_{\p_1}^* (T_b C_{\p_2}) 
\equiv  C_{\p_1}^* T_{\ol{b}}^* C_{\p_2} \equiv (T_{\ol{b}} C_{\p_1})^* C_{\p_2} \equiv 0 \modK. \]
For the second part, suppose $\lambda_1 \neq \lambda_2$ and note that $c(z) = \frac{z-\lambda_2}{\lambda_1-\lambda_2}$ is continuously differentiable in neighborhoods of $\lambda_1$ and $\lambda_2$ and satisfies $c(\lambda_1) = 1$ and $c(\lambda_2)=0$. Then by \cref{M_w C_p star formula},
$T_c C_{\p_1}^* \equiv C_{\p_1}^*$ and $T_{\ol{c}} C_{\p_2}^* \equiv 0$ (mod $\K$), and so
\[ C_{\p_2}C_{\p_1}^* \equiv  C_{\p_2}(T_c C_{\p_1}^*) 
\equiv  (C_{\p_2}^*)^* (T_{\ol{c}})^* C_{\p_1}^* \equiv (T_{\ol{c}} C_{\p_2}^*)^* C_{\p_1}^* \equiv 0 \modK. \]
\end{proof}

\begin{prop}\label{unique j}
Let $\p$ in $\A$ be such that $C_\p$ is essentially normal. Then $\p$ permutes the elements of $F(\p)$.
Furthermore, if $\p(\zeta) = \zeta'$ for some $\zeta, \zeta' \in F(\p)$, then $\p$ has equal order of contact, say $n$, with $\T$ at $\zeta$ and $\zeta'$, and we have
\[\dfrac{\p'(\zeta)^n}{\p'(\zeta')} = \dfrac{c}{c'},\]
where $c$ and $c'$ are non zero constants uniquely determined by $D_{n}(\p, \zeta)$ and $D_{n}(\p, \zeta')$ respectively. 
\end{prop}
\begin{proof} 
Let $F(\p) = \{ \zeta_1 = \zeta, ..., \zeta_r \}$ and let $\p_1, ..., \p_r$ be the basic functions with contact of order $n_1, ..., n_r$ with $\T$ at $\zeta_1, ..., \zeta_r$ respectively, guaranteed by \cref{sum decomposition into basics}, so that
\[ C_{\p} \equiv C_{\p_1} + ... + C_{\p_r} \modK. \]
For each $j$, we denote $\lambda_j = \p(\zeta_j)$, and let $\s_j$ be the unique branch of $(\p_j)_e^{-1}$ defined in some neighborhood of $\lambda_j$ which maps $\lambda_j$ to $\zeta_j$.

Let $b$ be a polynomial such that $b(\zeta_1) = 1$ and $b$ is $0$ at all the points in $\{ \zeta_2,... \zeta_r, \lambda_1,....,\lambda_r \} \setminus \{ \zeta_1\}$. Then by \cref{combinations of basic and adjoint}, \cref{M_w C_p formula} and \cref{M_w C_p star formula} we get that

\begin{align}\label{first equivalence}
M_b[C_\p, C_\p^*] &\equiv \sum_{j,k=1,...,r} M_b C_{\p_k}C_{\p_j}^* - M_b C_{\p_j}^*C_{\p_k} \nonumber \\
&\equiv \sum_{\substack{{k=1...r} \\ \lambda_k=\lambda_1 }} C_{\p_1}C_{\p_k}^* - \sum_{\substack{{j=1...r} \\ \lambda_j = \zeta_1 }} C_{\p_j}^*C_{\p_j} \modK.
\end{align}

We define a neighborhood $W$ of $\zeta_1$ by
\[ W = ( \bigcap_{\lambda_j = \zeta_1}W_j )\cap ( \bigcap_{\lambda_k = \lambda_1}\p_1^{-1}(W_k) ), \]
where $W_j = W(\p_j)$ denotes the neighborhood of $\lambda_j$ used in \cref{super nice adjoint formula mod compacts}, and let $\psi = \psi_{\zeta_1, \epsilon}$ with $\epsilon >0$ such that $\{ |z - \zeta_1| \leq \epsilon \} \cap \D$ is contained in $W$. Note that for each $j$ such that $\lambda_j = \zeta_1$, the map $\psi$ satisfies
\begin{itemize}
\item $\psi(\D) \subset W_j$;
\item $\psi$ is analytic at $\zeta_1$ and $\psi(\zeta_1) = \zeta_1$;
\item $\psi^{-1}(\{ \zeta_1 \}) = \{ \zeta_1 \}$.
\end{itemize}
Similarly, for each $k$ such that $\lambda_k = \lambda_1$, the map $\p_1 \circ \psi$ satisfies
\begin{itemize}
\item $(\p_1 \circ \psi)(\D) \subset \p_1(W) \subset W_k$;
\item $\p_1 \circ \psi$ is analytic at $\zeta_1$ and $(\p_1 \circ \psi)(\zeta_1) = \lambda_k$;
\item $(\p_1 \circ \psi)^{-1}(\{ \lambda_k \}) = \{ \zeta_1 \}$.
\end{itemize}

Using \cref{first equivalence} and applying \cref{super nice adjoint formula mod compacts} multiple times, we get
\begin{align*}
C_\psi M_b [C_\p, C_\p^*] 
&\equiv \sum_{\substack{{k=1...r} \\ \lambda_k=\lambda_1 }} C_{\p_1 \circ \psi}C_{\p_k}^* - \sum_{\substack{{j=1...r} \\ \lambda_j = \zeta_1 }} C_\psi C_{\p_j}^*C_{\p_j} \nonumber \\
&\equiv \sum_{\substack{{k=1...r} \\ \lambda_k=\lambda_1 }} \frac{1}{|\p'(\zeta_k)|}C_{\s_k \circ \p_1 \circ \psi}
- \sum_{\substack{{j=1...r} \\ \lambda_j = \zeta_1}} \frac{1}{|\p'(\zeta_j)|} C_{\p_j \circ \s_j \circ \psi}
\modK,
\end{align*}
where all the inducing maps $\s_k \circ \p_1 \circ \psi$ and $\p_j \circ \s_j \circ \psi$ above are in the class $\A$ with $F(\s_k \circ \p_1 \circ \psi) = F(\p_j \circ \s_j \circ \psi) = \{ \zeta_1 \}$. Note that the maps $\s_k \circ \p_1 \circ \psi$ above all map $\zeta_1$ to $\zeta_k$ and the maps $\p_j \circ \s_j \circ \psi$ above all fix $\zeta_1$.

Recalling that $C_\p$ is essentially normal, we see that 
\begin{align} \label{compact combination}
\sum_{\substack{{k=1...r} \\ \lambda_k=\lambda_1 }} \frac{1}{|\p'(\zeta_k)|}C_{\s_k \circ \p_1 \circ \psi}
- \sum_{\substack{{j=1...r} \\ \lambda_j = \zeta_1}} \frac{1}{|\p'(\zeta_j)|} C_{\p_j \circ \s_j \circ \psi}
\end{align}
is a compact linear combination of composition operators. Now \cref{first relation in K} combined with the properties of the inducing maps in \cref{compact combination} imply that $\lambda_k \neq \lambda_1$ for $k \neq 1$. By symmetry we conclude that $\lambda_1,...,\lambda_r$ are all distinct. Thus the first sum in \cref{compact combination} consists of exactly one term, and the second sum consists of at most one term. Again using \cref{first relation in K} we see that the second sum must also consist of exactly one term and so there is a unique $j_1$ such that $\lambda_{j_1} = \p(\zeta_{j_1}) = \zeta_1$. 
By symmetry, for any $k \in \{ 1, ..., r \}$ the exists a unique $j_k$ such that $\lambda_{j_k} = \p(\zeta_{j_k}) = \zeta_k$. Thus, $\p$ acts as a one-to-one map of $F(\p)$ onto itself. This proves the first statement.

Furthermore, \cref{compact combination} becomes
$\frac{1}{|\p'(\zeta_1)|}C_{\s_1 \circ \p_1 \circ \psi} - \frac{1}{|\p'(\zeta_{j_1})|} C_{\p_{j_1} \circ \s_{j_1} \circ \psi}$,
and so by \cref{first relation in K}, we have that $\s_1 \circ \p_1 \circ \psi$ and $\p_{j_1} \circ \s_{j_1} \circ \psi$ have equal order of contact, say $n$, at $\zeta_1$ and $D_n(\s_1 \circ \p_1 \circ \psi, \zeta_1) = D_n(\p_{j_1} \circ \s_{j_1} \circ \psi, \zeta_1 = \lambda_{j_1})$.
We use \cref{contact for composite} combined with \cref{order of contact of sigma} and the properties of $\psi$ to conclude that $\p$ has order of contact $n$ with $\T$ at both $\zeta_1$ and $\zeta_{j_1}$. We use \fdbf combined with invertiblity of $\psi$ in a neighborhood of $\zeta$ to conclude that $D_n (\s_1 \circ \p_1, \zeta_1) = D_n (\p_{j_1} \circ \s_{j_1}, \lambda_{j_1})$, and in particular $(\s_1 \circ \p_1)^{(n)}(\zeta_1) = (\p_{j_1} \circ \s_{j_1})^{(n)}(\lambda_{j_1})$. Now by \cref{s circ p} we get that
$\dfrac{c_1}{\p_1'(\zeta_1)} = \dfrac{c_{j_1}}{\p_{j_1}'(\zeta_{j_1})^n}$, where $c_j = \p_j^{(n_j)}(\zeta_j) - (\p_j)_e^{(n_j)}(\zeta_j)$ is a non zero constant determined by $D_{n_j}(\p, \zeta_j)$. By symmetry, the proof is complete.
\end{proof}

We now have the tools to prove our main theorem identifying the non trivially essentially normal composition operators induced by maps in $\A$.

\begin{thm}\label{ess normality criterion} Let $\p$ be in $\A$. Then $C_\p$ is non trivially essentially normal if and only if $F(\p) = \{ \zeta \}$ for some $\zeta \in \T$, $\p$ fixes $\zeta$, and $\p'(\zeta)=1$.
\end{thm}
\begin{proof} First note that for the case where $F(\p)$ is empty we have that $E(\p)$ is empty.
Thus, as noted in \cref{the class A}, $C_\p$ is compact and so trivially essentially normal in this case. For the case that $F(\p)$ contains one point, i.e.\ $| F(\p)| = 1$, we see that the statement is true by using \cref{sum decomposition into basics} to reduce to the basic case considered in \cref{basic essential normality}.

Now suppose that $|F(\p)| > 1$ and, in order to obtain a contradiction, suppose that $C_\p$ is essentially normal. By \cref{unique j} $\p$ permutes the points in $F(\p)$ and so decomposes $F(\p)$ into disjoint cycles. If $\zeta \in F(\p)$ is a fixed point, we have by \cref{unique j} that $\p'(\zeta)^{n-1} = 1$ and since $\p(\zeta) > 0$ in this case, we get that $\zeta$ is the unique Denjoy-Wolff point of $\p$. Thus there is at most one $\p$-cycle of length $1$ in $F(\p)$. Since $|F(\p)| > 1$, we must therefore have a $\p$-cycle $(\zeta_1,...,\zeta_k)$ in $F(\p)$ for some $k>1$.

By \cref{unique j} we get that $\p$ has equal order of contact $n$ with $\T$ at $\zeta_1, ..., \zeta_k$ and the following equations hold
\[ \dfrac{\p'(\zeta_1)^n}{\p'(\zeta_2)} = \dfrac{c_1}{c_2} , \qquad ... \qquad
\dfrac{\p'(\zeta_{k-1})^n}{\p'(\zeta_k)} = \dfrac{c_{k-1}}{c_k} , \qquad
\dfrac{\p'(\zeta_{k})^n}{\p'(\zeta_1)} = \dfrac{c_{k}}{c_1} ,
\]
where $c_j$ is a non zero constant uniquely determined by $D_n(\p, \zeta_j)$. Taking the product of the above equations yields
\[ \dfrac{\p'(\zeta_1)^n \cdot \p'(\zeta_2)^n \cdot ... \cdot \p'(\zeta_k)^n}
         {\p'(\zeta_2) \cdot ... \cdot \p'(\zeta_k) \cdot \p'(\zeta_1)} =
\dfrac{c_1 \cdot c_2 \cdot ... \cdot c_k}{c_2 \cdot ... \cdot c_k \cdot c_1},
\]
and so we get that
\begin{align} \label{power of product is 1}
(\p'(\zeta_1) \cdot ... \cdot \p'(\zeta_k))^{n-1} = 1.
\end{align}
Since $(\zeta_1, ..., \zeta_k)$ is a cycle under $\p$, for each $j = 1, ..., k$, the $k$-fold iterate $\p_{(k)}$ of $\p$ fixes $\zeta_j$ and satisfies  $\p_{(k)}'(\zeta_j) = \p'(\zeta_1) \cdot ... \cdot \p'(\zeta_k)$ by the chain rule. It follows that $\p_{(k)}'(\zeta_j) > 0$ and by \cref{power of product is 1} $(\p_{(k)}'(\zeta_j))^{n-1} = 1$. Thus $\p_{(k)}'(\zeta_j) = 1$ and $\zeta_1, ..., \zeta_k$ all qualify as the unique Denjoy-Wolff point of $\p_{(k)}$, and we have reached a contradiction.
\end{proof}

\begin{cor}\label{ess normality cor} Let $\p$ be a self-map of $\D$ which extends analytically to a neighborhood of $\ClD$. Then $C_\p$ is non-trivially essentially normal if and only if there exists $\zeta \in \T$ such that $\p$ fixes $\zeta$, $\p'(\zeta) = 1$, and $\p$ maps $\T \setminus \{ \zeta \}$ into $\D$.
\end{cor}
\begin{proof} We can assume that $\p$ is not linear fractional. If $\p$ is a finite Blaschke product of degree at least $2$, then $C_\p$ is not essentially normal by \cite{HAMADA}. Otherwise, by \cref{A is big}, $\p \in \A$. If $F(\p) = E(\p)$ is empty, then $C_\p$ is compact (as noted in \cref{the class A}) and so trivially essentially normal. Otherwise, the claim follows from \cref{ess normality criterion}.
\end{proof}

\subsection{Construction of Essentially Normal Composition Operators}
$ $

We combine the criterion for essential normality given in \cref{ess normality criterion} with several results from \cref{The CF Problem} to construct essentially normal composition operators which have arbitrary even order of contact with $\T$ at a point $\zeta$.

Recall from \cref{The CF Problem - section} that for a matrix $A = \begin{bmatrix} a_{11} & a_{12}\\ a_{21} & a_{22} \\ \end{bmatrix}$, we denote the corresponding linear fractional transformation by $L[A]$, so that $L[A]h = \dfrac{a_{11}h+a_{12}}{a_{21}h+a_{22}}$,
and use this notation to express the augmentation $f$ of $g$ by $a_0, a_1$ by $f(z) = L[A(a_0, a_1)(z)]g(z)$, where $A(a_0, a_1)(z)$ is defined by
\[ A(a_0, a_1)(z) = \begin{bmatrix} a_0 a_1 z &  - a_0 - a_1 z \\  a_1 z & -1 \\ \end{bmatrix}. \]
Additionally, recall from \cref{order of contact subsection} that for $\alpha \in \T$, the map $ \tau_\alpha \colon z \mapsto i\frac{\alpha-z}{\alpha+z}$ is a conformal map which we use to transfer $\D$ to $\U$.

\begin{thm} Let $\p$ be a rational self-map of $\D$ which extends analytically to a neighborhood of $\ClD$. Then $C_\p$ is non-trivially essentially normal if and only if $\p$ is of the form $\p = \tau_\zeta^{-1} \circ f \circ \tau_\zeta$ for some $\zeta \in \T$ with
\[ f(z) = L[A(0, 1)(z) A(s_1, t_1)(z) \cdot \cdot \cdot A(s_{m-1}, t_{m-1})(z)] w(z), \]
for some $n = 2m$, $s_1,..., s_{m-1} \in \R$, $t_1,...,t_{m-1} > 0$ and
$w$ a rational self-map of $\U$ which maps $\RH$ into $\U$.

Furthermore, this representation is unique, $\zeta \in \T$ is the fixed point of $\p$ and $n = 2m$ is the order of contact of $\p$ with $\T$ at $\zeta$.
\end{thm}

\begin{proof}
For the first direction, suppose that $C_\p$ is non-trivially essentially normal. By \cref{ess normality cor}, there exists $\zeta \in \T$ such that $\p$ fixes $\zeta$ and $\p'(\zeta) = 1$, and $\p$ maps $\T \setminus \{ \zeta \}$ into $\D$. 
Note that for $\zeta \neq \eta \in \T$, $f = \tau_\eta \circ \p \circ \tau_\eta^{-1}$ does not fix $0$ and so is not of the above form. The properties of $\p$ ensure that $f = \tau_\zeta \circ \p \circ \tau_\zeta^{-1}$ is a rational self-map of $\U$ which fixes $0$, satisfies $f'(0) = 1$ and maps $\RH \setminus \{ 0 \}$ into $\U$. Note that since $\p$ is analytic in a neighborhood of $\zeta$, $\p$ has order of contact $n = 2m$ with $\T$ at $\zeta$, and so $f$ has order of contact $n = 2m$ with $\R$ at $0$, for some positive integer $m$. By \cref{param of contact n} we see that $f$ has a unique representation in the form
\[ f(z) = L[A(s_0, t_0)(z) \cdots A(s_{m-1}, t_{m-1})(z)] w(z), \]
where the self-map $w$ of $\U$ satisfies $w(0) \in \U$ and $s_1,..., s_{m-1} \in \R$, $t_1,...,t_{m-1} > 0$ are uniquely determined by $f$. Note that $f(0)=0$ and $f'(0) = 1$, so that $s_0 = 0$ and $t_0 = 1$. Since $w$ can be obtained from $f$ by taking $m$ reductions, and by the observation in \cref{basic functions}, $w$ is a rational function mapping $\RH \setminus \{ 0 \}$ into $\U$.

For the second direction, let $\p$ be of the above form. By \cref{param of contact n}, $f$ has order of contact $n$ with $\R$ at $z=0$, and so $\p$ has order of contact $n$ with $\T$ at $\zeta$. Note that $f$ is obtained by taking $m$ augmentations of $w$ and so by the observation in \cref{basic functions} and our assumptions on $w$, we get that $f$ is rational and maps $\RH \setminus \{ 0 \}$ into $\U$. Since the last augmentation performed to obtain $f$ has parameters $0$ and $1$, $f$ satisfies $f(0) = 0$ and $f'(0) = 1$. The properties of $f$ ensure that $\p = \tau_\zeta^{-1} \circ f \circ \tau_\zeta$ is a rational self-map of $\D$ which extends analytically to a neighborhood of $\ClD$ and has order of contact $n$ with $\T$ at $\zeta$, fixes $\zeta$ and satisfies $\p'(\zeta)=1$. By \cref{ess normality cor}, the operator $C_\p$ is non-trivially essentially normal.
\end{proof}

Note that if $w$ is chosen to be a constant function in the formula above, then $\p$ is a degree $m$ rational self-map of $\D$ with order of contact $n = 2m$ with $\T$ at $\zeta$ which induces an essentially normal composition operator.

\paragraph{Example} Let $\zeta \in \T$, $s \in \R$, $t >0$ and $w \in \U$ and define $f(s, t, w)$ by
\begin{align*}
f(s, t, w)(z) &= L[A(0, 1)(z) A(s, t)(z)] w \\
&= L \left[ \begin{bmatrix} 0 & -z \\ z & -1 \end{bmatrix} \begin{bmatrix} s z & -s -t z \\ t z & -1 \end{bmatrix} \right] w =
\dfrac{-t z^2 w + z}{(s t z^2 - t z)w -s z - t z^2 +1}.
\end{align*}
Then for any $\zeta \in \T$, $\p(\zeta, s, t, w) = \tau_\zeta^{-1} \circ f(s, t, w) \circ \tau_\zeta$
has order of contact $4$ with $\T$ at $\zeta$ and induces a non trivially essentially normal composition operator. As a concrete example, we calculate $\p(1,0,1,i)(z) = \dfrac{z^2+2z+1}{z^2-2z+5}$.

\bibliographystyle{abbrv} \renewcommand{\bibname}{Bibliography} 
\bibliography{KatzEssNormalPreprint}

\end{document}